\documentclass
[
    a4paper,
    DIV=14,
    abstract=true,
    numbers=noenddot
]
{scrartcl}

\usepackage
{
    amsmath,
    amssymb,
    amsthm,
    authblk,
    dsfont, 
    enumitem,
    graphicx,
    mathtools,
    nicefrac,
    tikz,
    xcolor,
}

\usepackage[utf8]{inputenc}

\usepackage[pdffitwindow=false,
            plainpages=false,
            pdfpagelabels=true,
            pdfpagemode=UseOutlines,
            pdfpagelayout=SinglePage,
            bookmarks=false,
            colorlinks=true,
            hyperfootnotes=false,
            linkcolor=blue,
            urlcolor=blue!30!black,
            citecolor=green!50!black]{hyperref}

\usepackage[bf,normal]{caption}
\usepackage{color}

\graphicspath{{pics/}}

\DeclareMathAlphabet{\mathpzc}{OT1}{pzc}{m}{it}

\newcommand{\subfiguretitle}[1]{{\scriptsize{#1}} \\}
\newcommand{\R}{\mathbb{R}}                                       
\newcommand{\C}{\mathbb{C}}                                       
\newcommand{\pd}[2]{\frac{\partial#1}{\partial#2}}                
\newcommand{\innerprod}[2]{\left\langle #1,\, #2 \right\rangle}   
\newcommand{\ts}{\hspace*{0.1em}}                                 
\providecommand{\abs}[1]{\left\lvert #1 \right\rvert}             
\providecommand{\norm}[1]{\left\lVert #1 \right\rVert}            
\providecommand{\vdot}{\boldsymbol\cdot}                          

\newcommand\xqed[1]{\leavevmode\unskip\penalty9999 \hbox{}\nobreak\hfill \quad\hbox{#1}}
\newcommand{\exampleSymbol}{\xqed{$\triangle$}}

\DeclareMathOperator{\diag}{diag}
\DeclareMathOperator{\tr}{tr}
\DeclareMathOperator{\mspan}{span}

\let\div\relax
\DeclareMathOperator{\div}{div}

\newtheorem{theorem}{Theorem}[section]
\newtheorem{corollary}[theorem]{Corollary}
\newtheorem{lemma}[theorem]{Lemma}

\newtheorem{definition}[theorem]{Definition}
\theoremstyle{definition}
\newtheorem{example}[theorem]{Example}

\newtheorem{remark}[theorem]{Remark}
\newtheorem{talgorithm}[theorem]{Algorithm}

\makeatletter
\renewcommand*\env@matrix[1][*\c@MaxMatrixCols c]{%
  \hskip -\arraycolsep
  \let\@ifnextchar\new@ifnextchar
  \array{#1}}
\makeatother

\mathtoolsset{centercolon} 

\allowdisplaybreaks

\DeclareCaptionLabelFormat{period}{#1~#2}
\begin{document}

\title{Functional dynamic mode decomposition: \\ Learning infinite-dimensional systems from data}

\author[1]{Stefan Klus}
\author[2]{Eirini Ioannou}
\affil[1]{School of Mathematical \& Computer Sciences, Heriot--Watt University, Edinburgh, UK}
\affil[2]{Maxwell Institute for Mathematical Sciences, University of Edinburgh and Heriot--Watt University, Edinburgh, UK}

\date{}

\maketitle

\begin{abstract}
Dynamic mode decomposition (DMD) is a data-driven method that computes the best linear approximation of the underlying dynamical system and decomposes the dynamics into a superposition of characteristic spatiotemporal patterns. Originally introduced by the fluid dynamics community, DMD and its extensions have found widespread use in many other research areas such as molecular dynamics, climate science, engineering, finance, and neuroscience. Applications include dimensionality reduction, forecasting, system identification, control, and spectral clustering. In order to apply DMD to partial differential equations, the spatial domain is typically first discretized using finite difference or finite element techniques, thus implicitly rendering the problem finite-dimensional. We extend projected and exact DMD to infinite-dimensional systems. Rather than estimating matrices from vector-valued observations, our DMD variants learn finite-rank operators from functional data such as observables, densities, or wavefunctions. We show that conventional DMD algorithms can be regarded as special cases of their functional DMD counterparts. All results will be illustrated with the aid of guiding examples. We focus in particular on Koopman, Perron--Frobenius, and Koopman--von Neumann operators associated with graphons, ordinary differential equations, and stochastic differential equations.
\end{abstract}

\section{Introduction}

Functional data analysis \cite{LNVR07, Shang14, EH15, WCM16} focuses on the statistical analysis of data consisting of random functions---sampled from a potentially unknown distribution---and plays an important role in modern data science. The idea is to replace vectors by functions and matrices by compact linear operators. As a consequence, the data points are not contained in a finite-dimensional Euclidean space but rather an abstract infinite-dimensional Hilbert space. It has been shown that functional data analysis can outperform conventional multivariate statistics approaches by exploiting information contained in the functions themselves and their derivatives \cite{LNVR07, Shang14}. Many classical statistical methods for dimensionality reduction, linear regression, manifold learning, clustering, and classification have been extended to the functional data analysis setting \cite{WCM16}.

One of the most frequently used methods for the analysis of time-series data is dynamic mode decomposition (DMD) and its various extensions and nonlinear variants~\cite{Schmid10, TRLBK14, WKR15, KNPNCS20, KNKWKSN18}. These data-driven methods approximate transfer operators associated with the dynamical system, e.g., Koopman, Perron--Frobenius, or Koopman--von Neumann operators or the corresponding infinitesimal generators \cite{Ko31, KvN32, LaMa94, Mezic05}, which can then be used to compute spectral properties. Applications include the detection of metastable sets, forecasting, system identification, control, or spectral clustering \cite{BMM12, SS13, MauGon16, KKS16, MMS20, WuNo20, KNG26}. A detailed overview and further references can be found in \cite{KD24, Colbrook24}.

The goal of this work is to derive an extension of DMD that, given only functional data, learns a best-fit operator, i.e., an approximation of the propagator pertaining to an infinite-dimensional dynamical system, and its eigenvalues and eigenfunctions. This not only allows us to predict the dynamics, but also to analyze global properties of the system. We are in particular interested in detecting metastable sets and their implied timescales as well as identifying the underlying system itself. The main advantage of our approach is that it works directly with functional data defined on continuous domains without---depending on the representation of the snapshots---having to discretize it first.

In the last years, data-driven methodologies for analyzing infinite-dimensional dynamical systems using functional data have received little attention, with a few notable exceptions: The approach most closely related to our functional DMD framework is generalized EDMD~\cite{Mauroy21}, which aims to approximate Koopman operators associated with infinite-dimensional nonlinear dynamical systems using a dictionary containing basis functionals, whereas we directly represent the propagator in terms of the snapshots themselves, but focus mainly on linear dynamics. A special case of generalized EDMD, developed independently, is considered in \cite{OTY25}, where Koopman operators associated with random dynamical systems are estimated from distributional snapshot data. Similarly, a regression problem formulation involving Wasserstein distances between time-lagged distributional snapshots was used in \cite{KG20} to approximate the Perron--Frobenius operator. Instead of directly working with infinite-dimensional systems, another possibility is to first embed them into finite-dimensional spaces. In \cite{DHZ16}, embedding techniques are utilized to estimate finite-dimensional invariant sets of infinite-dimensional systems such as delay differential equations. An extension of this work is presented in \cite{ZDG18}, where finite-dimensional unstable manifolds of infinite-dimensional systems---in this case partial differential equations---are computed. Building on the aforementioned ideas, the paper \cite{PHNPSW25} leverages symmetries in PDEs and addresses the issues of unknown state spaces or partial functional measurements with the aid of Takens or Whitney embedding theorems. A comparatively different line of work, which can be regarded as an extension of SINDy to PDEs, focuses on identifying the terms appearing in the right-hand side of the PDE from typically spatially discretized snapshots \cite{RBPK17}. The main contributions of our work are:
\begin{enumerate}[leftmargin=4ex, itemsep=0ex, topsep=0.5ex, label=\roman*)]
\item We extend projected and exact DMD to infinite-dimensional systems. While projected functional DMD can be regarded as a Galerkin projection, exact functional DMD requires computing pseudoinverses of linear operators representing the training data.
\item We show that if we discretize the functional training data and represent it by finite-dimensional vectors, we obtain the classical DMD algorithms as special cases. Additionally, we compare functional DMD and generalized EDMD.
\item All results will be illustrated with guiding examples and benchmark problems ranging from random walks on graphons and Langevin dynamics to Koopman--von Neumann mechanics and the Kuramoto--Sivashinsky equation.
\end{enumerate}
The remainder of the paper is structured as follows: We will introduce projected and exact functional DMD and analyze its properties in Section~\ref{sec:Functional DMD}. Furthermore, we will explore relationships with conventional DMD algorithms and kernel-based methods. Section~\ref{sec:Applications} highlights different applications of functional DMD. Open problems and future work will be discussed in Section~\ref{sec:Conclusion}.

\section{Functional DMD}
\label{sec:Functional DMD}

In this section, we will derive two variants of DMD that work directly with functional data, analyze their properties, and compare the resulting algorithms with conventional DMD methods.

\subsection{Problem setting and training data}

Let $ \mathbb{H} $ be a separable Hilbert space with inner product $ \innerprod{\ts\cdot\ts}{\ts\cdot\ts} $ and induced norm $ \norm{\ts\cdot\ts} $. Furthermore, let $ \mathcal{W} \colon D(\mathcal{W}) \to \mathbb{H} $ be a linear operator, where $ D(\mathcal{W}) \subseteq \mathbb{H} $ denotes the domain of $ \mathcal{W} $. We will consider dynamical systems of the form
\begin{equation} \label{eq:dynamical system}
    \frac{\raisebox{-2pt}{$\partial$}}{\partial t} u(x, t) = \mathcal{W} \ts u(x, t),
\end{equation}
with initial condition $ u(x, 0) = u_0(x) $, where $ x \in \Omega \subseteq \R^d $. In our setting, $ \mathcal{W} $ could, for instance, be an integral or differential operator and $ \mathbb{H} $ a potentially weighted space of square-integrable functions, a Sobolev space, or a reproducing kernel Hilbert space. We assume that $ \mathcal{W} $ generates a strongly continuous semi-flow $ \big(\phi^t)_{t \ge 0} \colon \mathbb{H} \to \mathbb{H} $ such that
\begin{equation*}
    \phi^t\big(u_0\big) = e^{t \ts \mathcal{W}} u_0 = u(\ts\cdot\ts, t).
\end{equation*}
We will call $ e^{t \ts \mathcal{W}} $ the \emph{propagator} associated with the generator $ \mathcal{W} $. For a more detailed introduction, we refer to \cite{Mauroy21}. If $ \mu $ is an eigenvalue of the generator, then due to the spectral mapping theorem $ \lambda = e^{t \ts \mu} $ is, under assumptions detailed in \cite{Pazy83}, an eigenvalue of the corresponding propagator and the eigenfunctions are identical.

\begin{example} \label{ex:heat equation}
Let $ \Omega = (0, 1) $, $ \mathbb{H} = L^2(\Omega) $, and $ T > 0 $. Consider the heat equation
\begin{alignat*}{2}
    \frac{\raisebox{-2pt}{$\partial$}}{\partial t} u(x, t) &= \frac{\raisebox{-2pt}{$\partial^2$}}{\partial x^2} u(x, t) && \forall \ts (x,t) \in \Omega \times (0, T], \\
    u(x, 0) &= u_0(x) && \forall \ts x \in \Omega, \\
    u(0, t) &= 0, \; u(1, t) = 0 ~~ && \forall \ts t \in [0, T].
\end{alignat*}
The linear operator is in this case given by $ \mathcal{W} = \pd{^2}{x^2} $. Provided that $ u_0 $ is sufficiently smooth, we can write
\begin{equation*}
    u_0(x) = \sum_{k=1}^\infty \gamma_k \ts \sin(k \ts \pi \ts x)
    \implies
    u(x, t) = \sum_{k=1}^\infty \gamma_k \ts \sin(k \ts \pi \ts x) \ts e^{-k^2 \pi^2 t}. \tag*{\exampleSymbol}
\end{equation*}
\end{example}

In what follows, we will assume that we have measured or estimated the states $ u_i := u(\ts\cdot\ts, t_i) $ and $ v_i := \phi^\tau(u_i) = u(\ts\cdot\ts, t_i + \tau) $ of the system at $ m $ time points $ t_i $, where $ \tau $ is a fixed lag time. That is, our training data is given by $ \big\{ (u_i, v_i) \big\}_{i=1}^m $. The functions $ u_i $ and $ v_i $ could either be given by short simulations (or experiments), i.e., we generate $ m $ initial conditions $ u_i $ and measure $ v_i = \phi^\tau(u_i) $, or by one long simulation (or experiment), i.e., we select $ u_i = u(\ts\cdot\ts, (i-1) \ts \tau) $ and $ v_i = \phi^\tau(u_i) = u(\ts\cdot\ts, i \ts \tau) $. These two strategies can also be combined.

\subsection{Projected functional DMD}
\label{ssec:Projected functional DMD}

We first derive a variant of functional DMD from a Galerkin projection point of view, a more direct operator-based formulation will be considered later.

\subsubsection{Functional data vectors}

Similar to the data matrices required for DMD, we define arrays containing the training data. This simplifies the notation and facilitates comparisons between conventional DMD and its functional data analysis counterparts.

\begin{definition}[Functional data vectors]
We define the \emph{functional data vectors} $ U, V \in \mathbb{H}^{1 \times m} $ by
\begin{equation*}
    U =
    \begin{bmatrix}
        u_1 & \dots & u_m
    \end{bmatrix}
    \quad \text{and} \quad
    V =
    \begin{bmatrix}
        v_1 & \dots & v_m
    \end{bmatrix},
\end{equation*}
i.e., $ U $ and $ V $ are row-vectors comprising functions.
\end{definition}

We call the functions $ u_i $ contained in $ U $ the \emph{dictionary}, which spans an at most $ m $-dimensional subspace $ \mathbb{U} = \mspan\big\{ u_i \big\}_{i=1}^m $ of $ \mathbb{H} $. Analogously, we define $ \mathbb{V} = \mspan\big\{ v_i \big\}_{i=1}^m $. Vectors $ \alpha, \beta \in \C^m $ thus define functions $ f \in \mathbb{U} $ and $ g \in \mathbb{V} $ via
\begin{equation*}
    f = U \alpha := \sum_{i=1}^m \alpha_i \ts u_i
    \quad \text{and} \quad
    g = V \beta := \sum_{i=1}^m \beta_i \ts v_i.
\end{equation*}
We assume the functions $ u_i $ and $ v_i $ to be real-valued, but allow complex-valued coefficients here since the eigenvalues and eigenfunctions of the propagator will in general not be real-valued. Let $ C_{uu}, C_{uv} \in \R^{m \times m} $ be the Gram matrices defined by
\begin{equation*}
    \big[C_{uu}\big]_{ij} = \innerprod{u_i}{u_j}
    \quad \text{and} \quad
    \big[C_{uv}\big]_{ij} = \innerprod{u_i}{v_j}.
\end{equation*}
These matrices will be required for computing Galerkin projections and pseudoinverses of operators.

\subsubsection{Galerkin projection}
\label{sec:Galerkin projection}

Assume that the eigenfunction $ \varphi_\ell $ associated with the eigenvalue $ \lambda_\ell $ of the propagator $ e^{\tau \ts \mathcal{W}} $ is contained in $ \mathbb{U} $, i.e., there exist coefficients $ \xi^{(\ell)} \in \C^m $ such that
\begin{equation*}
    \varphi_\ell = U \xi^{(\ell)} = \sum_{i=1}^m \xi_i^{(\ell)} \ts u_i.
\end{equation*}
This implies that
\begin{equation*}
    e^{\tau \ts \mathcal{W}} \varphi_\ell
        = \sum_{i=1}^m \xi_i^{(\ell)} \ts e^{\tau \ts \mathcal{W}} \ts u_i
        = \sum_{i=1}^m \xi_i^{(\ell)} \ts v_i
        \stackrel{!}{=} \lambda_\ell \sum_{i=1}^m \xi_i^{(\ell)} \ts u_i = \lambda_\ell \ts \varphi_\ell.
\end{equation*}
Taking the inner product with the test function $ u_j $ on both sides, we have
\begin{equation*}
    \sum_{i=1}^m \xi_i^{(\ell)} \innerprod{v_i}{u_j} = \lambda_\ell \sum_{i=1}^m \xi_i^{(\ell)} \innerprod{u_i}{u_j}.
\end{equation*}
Collecting all equations for $ j = 1, \dots, m $, we finally obtain the generalized eigenvalue problem
\begin{equation*}
    C_{uv} \ts \xi^{(\ell)} = \lambda_\ell \ts C_{uu} \ts \xi^{(\ell)}.
\end{equation*}
Note that this is a standard Galerkin projection of the propagator $ e^{\tau \ts \mathcal{W}} $ onto $ \mathbb{U} $. Unless stated otherwise, we will assume that the functions $ u_i $ are linearly independent so that the matrix $ C_{uu} $ is invertible. We then obtain the eigenvalue problem $ A \ts \xi^{(\ell)} = \lambda_\ell \ts \xi^{(\ell)} $, with $ A = C_{uu}^{-1} \ts C_{uv} $. If $ C_{uu} $ is not invertible, we define $ A = C_{uu}^+ \ts C_{uv} $, where $ ^+ $ denotes the pseudoinverse. Inspired by the well-known and, as shown below, closely related projected DMD, we call our method \emph{projected functional DMD}.

\begin{talgorithm}[Projected functional DMD] \label{alg:projected functional DMD}
The projected functional DMD eigenfunctions $ \varphi_\ell $, with $ \ell = 1, \dots, m $, can be computed as follows:
\begin{enumerate}[leftmargin=4ex, itemsep=0ex, topsep=0.5ex]
\item Define $ A = C_{uu}^{-1} \ts C_{uv} $.
\item Determine the eigenvalues $ \lambda_\ell $ and eigenvectors $ \xi^{(\ell)} $ of $ A $.
\item Compute $ \varphi_\ell = U \xi^{(\ell)} $.
\end{enumerate}
\end{talgorithm}

\begin{example} \label{ex:heat equation projected}

\begin{figure}
    \centering
    \begin{minipage}[t]{0.32\linewidth}
        \centering
        \subfiguretitle{(a)}
        \includegraphics[height=3.9cm]{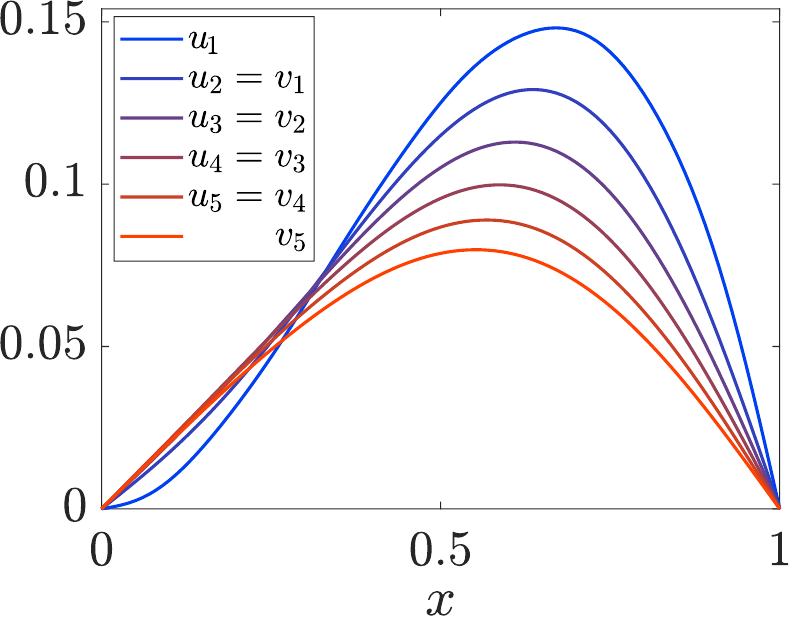}
    \end{minipage}
    \begin{minipage}[t]{0.32\linewidth}
        \centering
        \subfiguretitle{(b)}
        \vspace*{0.22ex}
        \includegraphics[height=3.89cm]{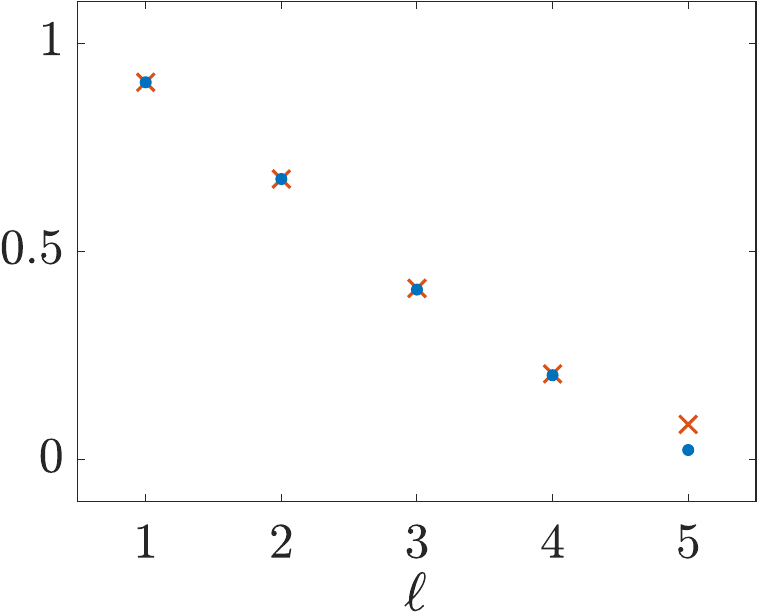}
    \end{minipage}
    \begin{minipage}[t]{0.32\linewidth}
        \centering
        \subfiguretitle{(c)}
        \vspace*{0.2ex}
        \includegraphics[height=3.85cm]{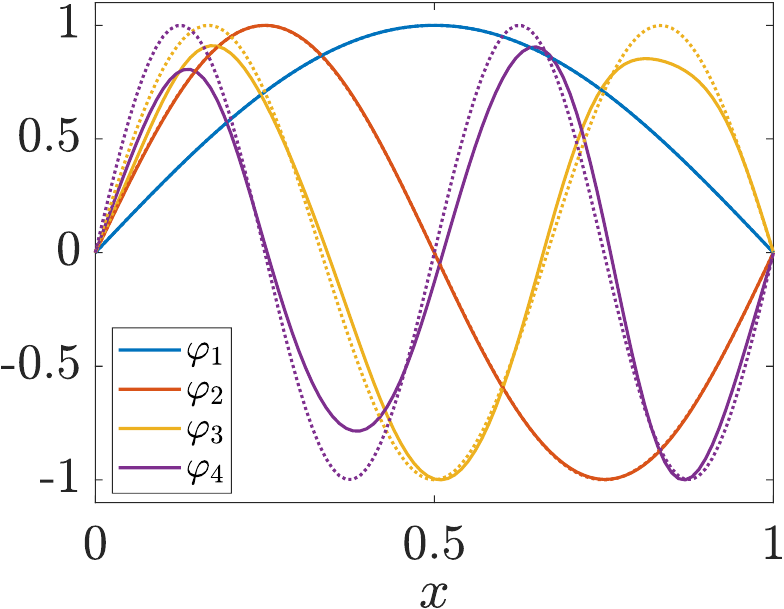}
    \end{minipage}
    \caption{(a)~Functions $ u_i $ and $ v_i $ given by solutions of the heat equation at different times $ t_i $ and $ t_i + \tau $, where $ t_i = (i-1) \ts \tau $. (b)~Comparison of the numerically computed eigenvalues (in blue) and the analytically computed eigenvalues $ e^{-\ell^2 \pi^2 \tau} $ (in red). (c)~First four numerically computed eigenfunctions. The dotted lines represent the true eigenfunctions $ \sin(\ell \ts \pi \ts x) $.}
    \label{fig:heat equation projected}
\end{figure}

Considering again the heat equation defined in Example~\ref{ex:heat equation}, let the initial condition $ u_0 $ be given by a series expansion with coefficients $ \gamma_k $. Using the orthogonality of the sine functions, it follows that
\begin{equation*}
    \big[C_{uu}\big]_{ij}
        = \int_0^1 \sum_{k=1}^\infty \gamma_k \ts \sin(k \ts \pi \ts x) \ts e^{-k^2 \pi^2 t_i} \sum_{l=1}^\infty \gamma_l \ts \sin(l \ts \pi \ts x) \ts e^{-l^2 \pi^2 t_j} \ts \mathrm{d}x
        = \frac{1}{2} \sum_{k=1}^\infty \gamma_k^2 \ts e^{-k^2 \pi^2 (t_i + t_j)}.
\end{equation*}
The entry $ \big[C_{uv}\big]_{ij} $ can be computed by replacing $ t_j $ by $ t_j + \tau $. We choose the initial condition
\begin{equation*}
    u_0(x) = x^2 \ts (1 - x)
    \quad \text{so that} \quad
    \gamma_k = \frac{(-1)^{k+1} \ts 8 - 4}{k^3 \ts \pi^3},
\end{equation*}
the lag time $ \tau = \frac{1}{100} $, and $ m = 5 $ snapshots. Defining $ t_i = (i-1) \ts \tau $, we obtain the functions $ u_i $ and $ v_i $ shown in Figure~\ref{fig:heat equation projected}\ts(a). We then compute the Gram matrices and apply Algorithm~\ref{alg:projected functional DMD}. A~comparison of the numerically and analytically computed eigenvalues can be found in Figure~\ref{fig:heat equation projected}\ts(b). The eigenvalues can be converted to frequencies via
\begin{equation*}
    \omega_\ell = \sqrt{ \frac{-\log(\lambda_\ell)}{\pi^2 \ts \tau} } \implies
    \omega_1 = 1.00,~
    \omega_2 = 2.00,~
    \omega_3 = 3.01,~
    \omega_4 = 4.02,~
    \omega_5 = 6.16.
\end{equation*}
The first four values are accurate estimates of the true frequencies. The corresponding eigenfunctions are shown in Figure~\ref{fig:heat equation projected}\ts(c). Although the training data set consists of just five pairs of functions, projected functional DMD determines good approximations of the leading eigenvalues and eigenfunctions. However, the higher the frequency, the less trustworthy the estimated eigenfunction. This is due to the fact that the coefficients $ \gamma_k $ decay rapidly for increasing~$ k $, which makes them more difficult to estimate. Furthermore, higher frequencies are smoothed out quickly by the dynamics. More accurate estimates can be obtained by increasing $ m $ as well as by using multiple initial conditions. The lag time $ \tau $ also plays an important role. \exampleSymbol
\end{example}

The eigenvalues and eigenfunctions associated with the one-dimensional heat equation are of course well-known. The example just illustrates how spectral properties can be estimated from functional time-series data.

\begin{remark}
We would like to point out that:
\begin{enumerate}[leftmargin=4ex, itemsep=0ex, topsep=0.5ex, label=\roman*)]
\item If we estimate the propagator from one long trajectory, we implicitly have to assume that the initial condition excites all the eigenfunctions we are interested in. Choosing $ u_0 = \varphi_\ell $ or a linear combination of a few eigenfunctions, we will not be able to recover any other eigenfunctions. Generating multiple different initial conditions will in general mitigate this problem.
\item Depending on the method we use for solving \eqref{eq:dynamical system} or measuring the state of the system, the training data could be represented in different ways. If we, for example, apply spectral methods, the functions $ u_i $ and $ v_i $ could be given by Fourier series or Chebyshev polynomials. Alternatively, we could use kernel density estimates or simply a grid or finite element discretization of the spatial domain.
\end{enumerate}
\end{remark}

\begin{example} \label{ex:KDE and grid discretization}
Let us consider two different data representations:
\begin{enumerate}[leftmargin=4ex, itemsep=0ex, topsep=0.5ex, label=\roman*)]
\item Given a reproducing kernel Hilbert space $ \mathbb{H} $ defined by a symmetric positive definite kernel $ k \colon \R^d \times \R^d \to \R $, we represent the functions $ u_i $ and $ v_j $ by kernel density estimates, i.e.,
\begin{equation*}
    u_i = \frac{1}{n} \sum_{\imath=1}^n k\big(\ts\cdot\ts, x_\imath^{(i)}\big)
    \quad \text{and} \quad
    v_j = \frac{1}{n} \sum_{\jmath=1}^n k\big(\ts\cdot\ts, y_\jmath^{(j)}\big),
\end{equation*}
where $ x_\imath^{(i)} $ and $ y_\jmath^{(j)} $ are given samples at times $ t_i $ and $ t_j + \tau $, respectively. It then follows that
\begin{align*}
    \big[C_{uu}\big]_{ij}
        &= \innerprod{u_i}{u_j}
        = \frac{1}{n^2} \sum_{\imath=1}^n \sum_{\jmath=1}^n \big\langle k\big(\ts\cdot\ts, x_\imath^{(i)}\big), k\big(\ts\cdot\ts, x_\jmath^{(j)}\big) \big\rangle
        = \frac{1}{n^2} \sum_{\imath=1}^n \sum_{\jmath=1}^n k\big(x_\imath^{(i)}, x_\jmath^{(j)}\big), \\
    \big[C_{uv}\big]_{ij}
        &= \innerprod{u_i}{v_j}
        = \frac{1}{n^2} \sum_{\imath=1}^n \sum_{\jmath=1}^n \big\langle k(\ts\cdot\ts, x_\imath^{(i)}), k\big(\ts\cdot\ts, y_\jmath^{(j)}\big) \big\rangle
        = \frac{1}{n^2} \sum_{\imath=1}^n \sum_{\jmath=1}^n k\big(x_\imath^{(i)}, y_\jmath^{(j)}\big).
\end{align*}
\item Assume we discretize the spatial domain using a regular grid comprising $ n \gg m $ points $ x_1, \dots, x_n $ and approximate the functions $ u_i $ and $ v_i $ by column vectors $ \mathbf{u}_i, \mathbf{v}_i \in \R^n $, i.e.,
\begin{equation*}
    \mathbf{u}_i =
    \begin{bmatrix}
        u(x_1, t_i) & \dots & u(x_n, t_i)
    \end{bmatrix}^\top
    \quad \text{and} \quad
    \mathbf{v}_i =
    \begin{bmatrix}
        u(x_1, t_i + \tau) & \dots & u(x_n, t_i + \tau)
    \end{bmatrix}^\top,
\end{equation*}
then $ U, V \in \R^{n \times m} $ and the Gram matrices are given by $ C_{uu} = U^\top \ts U $ and $ C_{uv} = U^\top \ts V $. We thus obtain $ A = \big(U^\top U\big)^{-1} U^\top V = U^+ V \in \R^{m \times m} $. DMD, on the other hand, computes eigenvalues and eigenvectors of the matrix $ B = V \ts U^+ \in \R^{n \times n} $. The nonzero eigenvalues of $ A $ and $ B $ are identical. Given $ B \ts \eta^{(\ell)} = \lambda_\ell \ts \eta^{(\ell)} $ for $ \lambda_\ell \ne 0 $, define $ \xi^{(\ell)} = U^+ \eta^{(\ell)} $, then
\begin{equation*}
    A \ts \xi^{(\ell)}
        = \big(U^+ V\big) \ts U^+ \eta^{(\ell)}
        = U^+ B \ts \eta^{(\ell)}
        = \lambda_\ell \ts U^+ \eta^{(\ell)}
        = \lambda_\ell \ts \xi^{(\ell)}
\end{equation*}
and the corresponding eigenvector is given by
\begin{equation*}
    \varphi_\ell = U \xi^{(\ell)} = U \ts U^+ \eta^{(\ell)},
\end{equation*}
which is a projection of $ \eta^{(\ell)} $ onto the span of the columns of $ U $. This illustrates the close relationship between DMD and functional DMD for discrete data. In fact, our method is in this case equivalent to what is called \emph{standard DMD} or \emph{projected DMD} in \cite{TRLBK14}, while the DMD variant that computes eigenvalues and eigenvectors of $ B $ is referred to as \emph{exact DMD}. Note that although $ B $ is of size $ n \times n $, the exact DMD algorithm avoids constructing the full matrix and computes the eigenvectors of a projected matrix instead, which are then mapped back to the original state space. We will derive exact functional DMD below. For a detailed introduction to DMD, see \cite{Schmid10, TRLBK14, KNKWKSN18, Colbrook24}. \exampleSymbol
\end{enumerate}
\end{example}

Example~\ref{ex:heat equation projected} and Example~\ref{ex:KDE and grid discretization} show that, depending on the representation of the functions and the inner product, we might not need to explicitly compute or approximate the integrals required for $ C_{uu} $ and $ C_{uv} $ using numerical integration techniques. Even if we only have gridded data, we could replace the Euclidean inner product between the vector representations of the functions employed by conventional DMD algorithms by more accurate numerical quadrature rules for the approximation of the Gram matrices.

\subsubsection{Best approximation}

We will now show that, given only the training data detailed above, the derived operator representation is indeed the best approximation.

\begin{lemma} \label{lem:inner product}
Given two functions $ f = U \alpha \in \mathbb{U} $ and $ g = U \beta \in \mathbb{U} $, with $ \alpha, \beta \in \C^m $, it holds that $ \innerprod{f}{g} = \alpha^\top C_{uu} \ts \overline{\beta} $. Furthermore, $ \innerprod{f}{v_k} = \alpha^\top \big[C_{uv}\big]_{:, k} $, where $ \big[C_{uv}\big]_{:, k} $ is the $ k $th column of $ C_{uv} $.
\end{lemma}

\begin{proof}
Using the sesquilinearity of the inner product, we have
\begin{equation*}
    \innerprod{f}{g}
        = \innerprod{U \alpha}{U \beta}
        = \sum_{i=1}^m \sum_{j=1}^m \alpha_i \ts \overline{\beta}_j \innerprod{u_i}{u_j}
        = \alpha^\top C_{uu} \ts \overline{\beta}.
\end{equation*}
Similarly,
\begin{equation*}
    \innerprod{f}{v_k}
        = \innerprod{U \alpha}{v_k}
        = \sum_{i=1}^m \alpha_i \innerprod{u_i}{v_k}
        = \alpha^\top C_{uv} \ts e_k,
\end{equation*}
where $ e_k $ is the $ k $th unit vector.
\end{proof}

\begin{definition}[Projected operator]
Given a matrix $ A = \big[a_{ij}\big]_{i,j=1}^m \in \R^{m \times m} $ and a function $ f = U \alpha \in \mathbb{U} $, we define the operator $ \mathcal{A} \colon \mathbb{U} \to \mathbb{U} $ by
\begin{equation*}
    \mathcal{A} f = U (A \ts \alpha).
\end{equation*}
\end{definition}

Our goal is to determine the matrix $ A $ in such a way that it minimizes the prediction error for the training data.

\begin{theorem}
Let $ U $ and $ V $ be as defined above, then the optimal solution of the minimization problem
\begin{equation*}
    \min_{A \in \R^{m \times m}} \sum_{i=1}^m \norm{\mathcal{A} \ts u_i - v_i}^2
\end{equation*}
is given by $ A = C_{uu}^{-1} \ts C_{uv} $.
\end{theorem}

\begin{proof}
It holds that
\begin{equation*}
    \norm{\mathcal{A} \ts u_i - v_i}^2
        = \innerprod{\mathcal{A} \ts u_i - v_i}{\mathcal{A} \ts u_i - v_i}
        = \innerprod{\mathcal{A} \ts u_i}{\mathcal{A} \ts u_i} - 2 \innerprod{\mathcal{A} \ts u_i}{v_i} + \innerprod{v_i}{v_i}.
\end{equation*}
We first compute $ \mathcal{A} \ts u_i = U (A \ts e_i) = U A_{:, i} $. Using Lemma~\ref{lem:inner product}, this implies
\begin{equation*}
    \innerprod{\mathcal{A} \ts u_i}{\mathcal{A} \ts u_i}
        = A_{:,i}^\top \ts C_{uu} \ts A_{:,i}
\end{equation*}
and
\begin{equation*}
    \innerprod{\mathcal{A} \ts u_i}{v_i}
        = A_{:,i}^\top \big[C_{uv}\big]_{:,i}.
\end{equation*}
We can ignore the third term since it is independent of $ A $ and does not affect the solution of the optimization problem. Summing over $ i $, this yields
\begin{align*}
    \min_{A \in \R^{m \times m}} \sum_{i=1}^m \norm{\mathcal{A} \ts u_i - v_i}^2
        &= \min_{A \in \R^{m \times m}} \tr\big( A^\top C_{uu} \ts A \big) - 2 \ts \tr\big( A^\top C_{uv} \big).
\end{align*}
Computing the derivative with respect to the matrix $ A $ and setting it to zero, we finally obtain $ 2 \ts C_{uu} \ts A - 2 \ts C_{uv} = 0 $ and hence, assuming $ C_{uu} $ is invertible, $ A = C_{uu}^{-1} \ts C_{uv} $.
\end{proof}

\subsubsection{Spectral decomposition and forecasting}

Given $ A \ts \xi^{(\ell)} = \lambda_\ell \ts \xi^{(\ell)} $, we define $ \varphi_\ell = U \xi^{(\ell)} $, which implies
\begin{equation*}
    \mathcal{A} \ts \varphi_\ell = U \big(A \ts \xi^{(\ell)}\big) = \lambda_\ell \ts U \xi^{(\ell)} = \lambda_\ell \ts \varphi_\ell.
\end{equation*}
That is, we can compute eigenvalues and eigenfunctions of the operator $ \mathcal{A} $ by computing eigenvalues and eigenvectors of the matrix $ A $. This is consistent with the Galerkin projection derived above. Constructing the matrices $ \Xi = \big[\xi^{(1)}, \dots, \xi^{(m)}\big] $ and $ \Lambda = \diag(\lambda_1, \dots, \lambda_m) $, we have $ A = \Xi \ts \Lambda \ts \Xi^{-1} $. For a function $ f = U \alpha $, we can thus write
\begin{equation*}
    \mathcal{A} f
        = U \big(A \ts \alpha) = U \big(\Xi \ts \Lambda \ts \underbrace{\Xi^{-1} \ts \alpha}_{=: z})
        = \sum_{\ell=1}^m \lambda_\ell \ts z_\ell \ts U \xi^{(\ell)}
        = \sum_{\ell=1}^m \lambda_\ell \ts z_\ell \ts \varphi_\ell,
\end{equation*}
which then implies
\begin{equation*}
    \mathcal{A}^p f = \sum_{\ell=1}^m \lambda_\ell^p \ts z_\ell \ts \varphi_\ell.
\end{equation*}
If $ f $ is the initial condition at time $ t = 0 $, then $ \mathcal{A}^p f $ is an approximation of the solution at time $ t = p \ts \tau $. By solving the system of linear equations $ \Xi \ts z = \alpha $, we can express the evolution of the dynamical system in terms of the eigenvalues $ \lambda_\ell $, eigenfunctions $ \varphi_\ell $, and modes $ z_\ell $. One limitation though is that only initial conditions contained in $ \mathbb{U} $ are supported.

\subsection{Exact functional DMD}

As shown above, by discretizing a function and representing it as a vector, we obtain projected DMD as a special case. Our goal now is to derive a variant of functional DMD that is more closely related to exact DMD.

\subsubsection{Functional data operators}

We can also interpret the functional data vectors as operators, which then allows us to compute singular value decompositions and pseudoinverses.

\begin{definition}[Functional data operators]
Let $ \alpha, \beta \in \C^ m $. We define the \emph{functional data operators} $ \mathcal{U} \colon \C^m \to \mathbb{H} $ and $ \mathcal{V} \colon \C^m \to \mathbb{H} $ by
\begin{equation*}
    \mathcal{U} \alpha = U \alpha = \sum_{i=1}^m \alpha_i \ts u_i
    \quad \text{and} \quad
    \mathcal{V} \beta = V \beta = \sum_{i=1}^m \beta_i \ts v_i.
\end{equation*}
\end{definition}

\begin{lemma}
The adjoint $ \mathcal{U}^* \colon \mathbb{H} \to \C^m $ is given by
\begin{equation*}
    \mathcal{U}^* g =
    \begin{bmatrix}
        \innerprod{g}{u_1} \\
        \vdots \\
        \innerprod{g}{u_m}
    \end{bmatrix}.
\end{equation*}
\end{lemma}

\begin{proof}
Given $ \alpha \in \C^m $ and $ g \in \mathbb{H} $, we have
\begin{equation*}
    \innerprod{\,\mathcal{U} \alpha}{g}_\mathbb{H}
        = \sum_{i=1}^m \alpha_i \innerprod{u_i}{g}_\mathbb{H}
        = \sum_{i=1}^m \alpha_i \overline{\innerprod{g}{u_i}}_\mathbb{H}
        = \innerprod{\alpha}{\mathcal{U}^* g}_{\C^m}. \qedhere
\end{equation*}
\end{proof}

Choosing $ g = U \beta $, this implies $ \innerprod{\,\mathcal{U} \alpha}{U \beta}_\mathbb{H} = \innerprod{\alpha}{\mathcal{U}^* U \beta}_{\C^m} = \innerprod{\alpha}{C_{uu} \ts \beta}_{\C^m} = \alpha^\top C_{uu} \ts \overline{\beta} $.

\subsubsection{Singular value decomposition and pseudoinverse}

Exact DMD requires computing singular value decompositions and pseudoinverses of data matrices. In order to extend this to the functional data analysis setting, we need to compute singular value decompositions and pseudoinverses of functional data operators. For details on spectral decompositions and generalized inverses of bounded linear operators, see, e.g., \cite{EHN96, EH15, MSKS20}.

\begin{definition}[Rank-one operator]
Given two Hilbert spaces $ \mathbb{H}_1 $ and $ \mathbb{H}_2 $ (here, $ \C^m $ and $ \mathbb{H} $, not necessarily in that order) and nonzero elements $ s \in \mathbb{H}_1 $ and $ r \in \mathbb{H}_2 $, we define the linear \emph{rank-one operator} $ r \otimes s \colon \mathbb{H}_1 \to \mathbb{H}_2 $ by
\begin{equation*}
    (r \otimes s) h = \innerprod{h}{s} r.
\end{equation*}
\end{definition}

\begin{lemma}
Let $ C_{uu} = \Theta \ts \Sigma^2 \ts \Theta^\top $, where $ \Theta = [\theta^{(1)}, \dots, \theta^{(m)}] $ contains the eigenvectors and $ \Sigma^2 = \diag(\sigma_1^2, \dots, \sigma_m^2) $ the eigenvalues. The singular value decomposition of the operator $ \mathcal{U} \colon \C^m \to \mathbb{H} $ is then given by
\begin{equation*}
    \mathcal{U} = \sum_{\ell=1}^m \sigma_\ell \ts (r_\ell \otimes s_\ell),
\end{equation*}
with $ r_\ell = \frac{1}{\sigma_\ell} U \ts \theta^{(\ell)} $ and $ s_\ell = \theta^{(\ell)} $.
\end{lemma}

\begin{proof}
We compute the eigendecomposition of the operator $ \mathcal{U}^* \mathcal{U} \colon \C^m \to \C^m $, defined by
\begin{equation*}
    \mathcal{U}^* \mathcal{U} \alpha = \mathcal{U}^* \sum_{i=1}^m \alpha_i \ts u_i =
    \begin{bmatrix}
        \sum_{i=1}^m \alpha_i \innerprod{u_1}{u_i} \\
        \vdots \\
        \sum_{i=1}^m \alpha_i \innerprod{u_m}{u_i}
    \end{bmatrix}
    = C_{uu} \ts \alpha.
\end{equation*}
That is, the eigenvalues and eigenfunctions of the operator $ \mathcal{U}^* \mathcal{U} $ are the eigenvalues $ \sigma_\ell^2 $ and eigenvectors $ \theta^{(\ell)} $ of the symmetric positive definite matrix $ C_{uu} $. The singular values of $ \mathcal{U} $ are thus $ \sigma_\ell $ and the right singular functions are $ s_\ell = \theta^{(\ell)} $. The corresponding left singular functions can be computed via $ r_\ell = \frac{1}{\sigma_\ell} \ts \mathcal{U} s_\ell = \frac{1}{\sigma_\ell} \ts U s_\ell $.
\end{proof}

\begin{corollary} \label{cor:pseudoinverse}
The \emph{pseudoinverse} or \emph{Moore--Penrose inverse} $ \mathcal{U}^+ \colon \mathbb{H} \to \C^m $ is defined by
\begin{equation*}
    \mathcal{U}^+ = \sum_{\ell=1}^m \frac{1}{\sigma_\ell} \ts (s_\ell \otimes r_\ell).
\end{equation*}
Furthermore, given functions $ f = U \alpha $ and $ g = V \beta $, it holds that $ \mathcal{U}^+ f = \alpha $ and $ \mathcal{U}^+ g = C_{uu}^{-1} \ts C_{uv} \ts \beta $.
\end{corollary}

\begin{proof}
For an arbitrary function $ f \in \mathbb{H} $, it holds that
\begin{align*}
    \mathcal{U}^+ f
        &= \sum_{\ell=1}^m \frac{1}{\sigma_\ell} \ts (s_\ell \otimes r_\ell) \ts f
        = \sum_{\ell=1}^m \frac{1}{\sigma_\ell^2} \ts \big(\theta^{(\ell)} \otimes U \ts \theta^{(\ell)}\big) \ts f \\
        &= \sum_{\ell=1}^m \frac{1}{\sigma_\ell^2} \ts \big(\theta^{(\ell)} \otimes \theta^{(\ell)}\big)
        \begin{bmatrix}
            \innerprod{f}{u_1} \\
            \vdots \\
            \innerprod{f}{u_m}
        \end{bmatrix}
        = C_{uu}^{-1}
        \begin{bmatrix}
            \innerprod{f}{u_1} \\
            \vdots \\
            \innerprod{f}{u_m}
        \end{bmatrix}
\end{align*}
since $ C_{uu}^{-1} = \Theta \ts \Sigma^{-2} \ts \Theta^\top $. For functions of the form $ f = U \alpha $ and $ g = V \beta $, we have
\begin{equation*}
    \innerprod{f}{u_j} = \sum_{i=1}^m \innerprod{u_i}{u_j} \alpha_i = \big[C_{uu} \ts \alpha\big]_j
    \quad \text{and} \quad
    \innerprod{g}{u_j} = \sum_{i=1}^m \innerprod{v_i}{u_j} \beta_i = \big[C_{uv} \ts \beta\big]_j,
\end{equation*}
which implies that $ \mathcal{U}^+ f = \alpha $ and $ \mathcal{U}^+ g = C_{uu}^{-1} \ts C_{uv} \ts \beta $. It then follows that
\begin{equation*}
    \mathcal{U}^+ \mathcal{U} \alpha = \mathcal{U}^+ U \alpha = \alpha,
\end{equation*}
i.e., $ \mathcal{U}^+ \mathcal{U} = \mathcal{I}_{\C^m} $, so that $ \mathcal{U} \ts \mathcal{U}^+ \mathcal{U} = \mathcal{U} $ and $ \mathcal{U}^+ \mathcal{U} \ts \mathcal{U}^+ = \mathcal{U}^+ $. Since we assumed the functions $ u_i $ to be linearly independent, $ \mathcal{U} $ has a trivial null space. Furthermore, the range of $ \mathcal{U} $ is $ \mathbb{U} $. The operator $ \mathcal{U} \ts \mathcal{U}^+ $ projects any function $ f $ onto $ \mathbb{U} $ and is idempotent and self-adjoint.
\end{proof}

\subsubsection{Best-fit operator and forecasting}

Mirroring the definition of exact DMD, we now define a functional DMD variant that computes eigenfunctions of the operator $ \widetilde{\mathcal{A}} := \mathcal{V} \ts \mathcal{U}^+ \colon \mathbb{H} \to \mathbb{H} $. Note in particular that if the functions $ u_i $ are linearly independent, then $ \widetilde{\mathcal{A}} u_i = \mathcal{V} \ts \mathcal{U}^+ u_i = \mathcal{V} \ts \mathcal{U}^+ \mathcal{U} e_i = \mathcal{V} e_i = v_i $ so that
\begin{equation*}
    \sum_{i=1}^m \big\|\widetilde{\mathcal{A}} u_i - v_i\big\|_\mathbb{H} = 0.
\end{equation*}

\begin{talgorithm}[Exact functional DMD] \label{alg:exact functional DMD}
In order to compute the exact functional DMD eigenfunctions $ \widetilde{\varphi}_\ell $, we carry out the following steps:
\begin{enumerate}[leftmargin=4ex, itemsep=0ex, topsep=0.5ex]
\item Define $ \widetilde{A} = \Sigma^{-1} \ts \Theta^\top C_{uv} \ts \Theta \ts
\Sigma^{-1} $.
\item Determine the eigenvalues $ \widetilde{\lambda}_\ell $ and eigenvectors $ \widetilde{\xi}^{(\ell)} $ of $ \widetilde{A} $.
\item Compute $ \widetilde{\varphi}_\ell = \frac{1}{\widetilde{\lambda}_\ell \rule{0pt}{1.9ex}} V \ts \Theta \ts \Sigma^{-1} \widetilde{\xi}^{(\ell)} $.
\end{enumerate}
\end{talgorithm}

\begin{theorem}
The functions $ \widetilde{\varphi}_\ell $ computed in Algorithm~\ref{alg:exact functional DMD} are indeed eigenfunctions of $ \widetilde{\mathcal{A}} $.
\end{theorem}

\begin{proof}
Using Corollary~\ref{cor:pseudoinverse}, we have
\begin{equation*}
    \mathcal{U}^+ V \ts \Theta \Sigma^{-1} \widetilde{\xi}^{(\ell)} = C_{uu}^{-1} \ts C_{uv} \ts \Theta \Sigma^{-1} \widetilde{\xi}^{(\ell)}.
\end{equation*}
It then follows that
\begin{equation*}
    \widetilde{\mathcal{A}} \widetilde{\varphi}_\ell
        = \frac{1}{\widetilde{\lambda}_\ell \rule{0pt}{2.5ex}} \mathcal{V} \ts \mathcal{U}^+ V \ts \Theta \Sigma^{-1} \widetilde{\xi}^{(\ell)}
        = \frac{1}{\widetilde{\lambda}_\ell \rule{0pt}{2.5ex}} V \underbrace{\Theta \ts \Sigma^{-2} \ts \Theta^\top}_{C_{uu}^{-1}} C_{uv} \ts \Theta \ts \Sigma^{-1} \widetilde{\xi}^{(\ell)}
        = V \Theta \ts \Sigma^{-1} \widetilde{\xi}^{(\ell)}
        = \widetilde{\lambda}_\ell \ts \widetilde{\varphi}_\ell. \qedhere
\end{equation*}
\end{proof}

\begin{example}
If we again discretize the spatial domain and represent the functions $ u_i $ and $ v_i $ by $ n $-dimensional vectors $ \mathbf{u}_i $ and $ \mathbf{v}_i $ as described in Example~\ref{ex:KDE and grid discretization}, then exact functional DMD reduces to exact DMD. This can be seen as follows: Let $ U, V \in \R^{n \times m} $ be the data matrices and $ U = R \ts \Sigma \ts S^\top $ the compact singular value decomposition of $ U $. Thus, $ C_{uu} = S \ts \Sigma^2 \ts S^\top $ and $ \Theta = S $ so that
\begin{equation*}
    \widetilde{A}
        = \Sigma^{-1} \ts S^\top C_{uv} \ts S \ts \Sigma^{-1}
        = \Sigma^{-1} \ts S^\top U^\top V \ts S \ts \Sigma^{-1}
        = R^\top V S \ts \Sigma^{-1}
\end{equation*}
and $ \widetilde{\varphi}_\ell = \frac{1}{\widetilde{\lambda}_\ell \rule{0pt}{1.9ex}} V \ts S \ts \Sigma^{-1} \widetilde{\xi}^{(\ell)} $, which is the exact DMD algorithm presented in \cite{TRLBK14}. \exampleSymbol
\end{example}

We can now use the learned operator $ \widetilde{\mathcal{A}} $ or its spectral decomposition to predict the evolution of the dynamical system as described in Section~\ref{ssec:Projected functional DMD}.

\begin{example}

\begin{figure}
    \centering
    \begin{minipage}[t]{0.32\linewidth}
        \centering
        \subfiguretitle{(a)}
        \includegraphics[height=3.9cm]{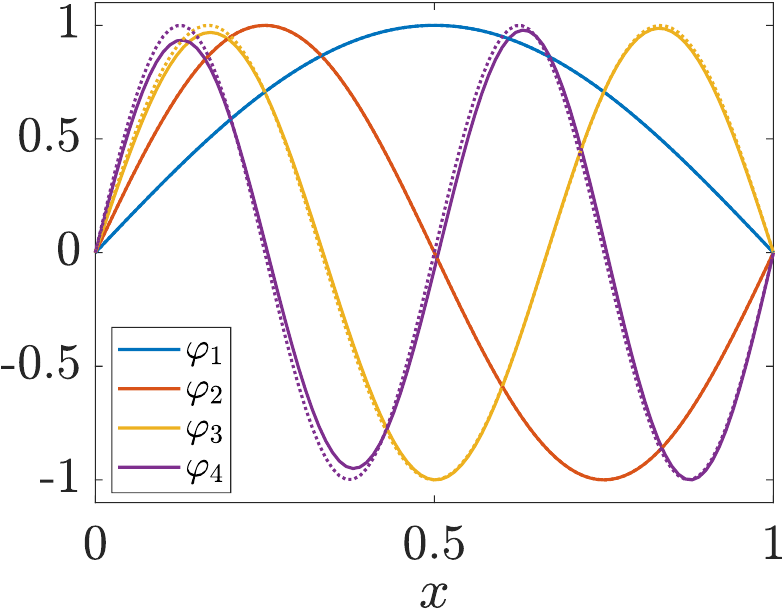}
    \end{minipage}
    \begin{minipage}[t]{0.32\linewidth}
        \centering
        \subfiguretitle{(b)}
        \includegraphics[height=3.91cm]{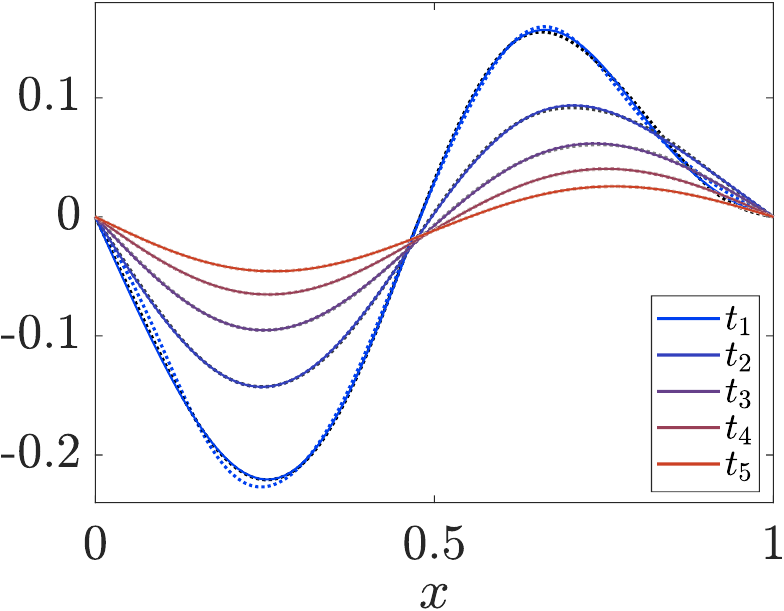}
    \end{minipage}
    \begin{minipage}[t]{0.32\linewidth}
        \centering
        \subfiguretitle{(c)}
        \includegraphics[height=3.91cm]{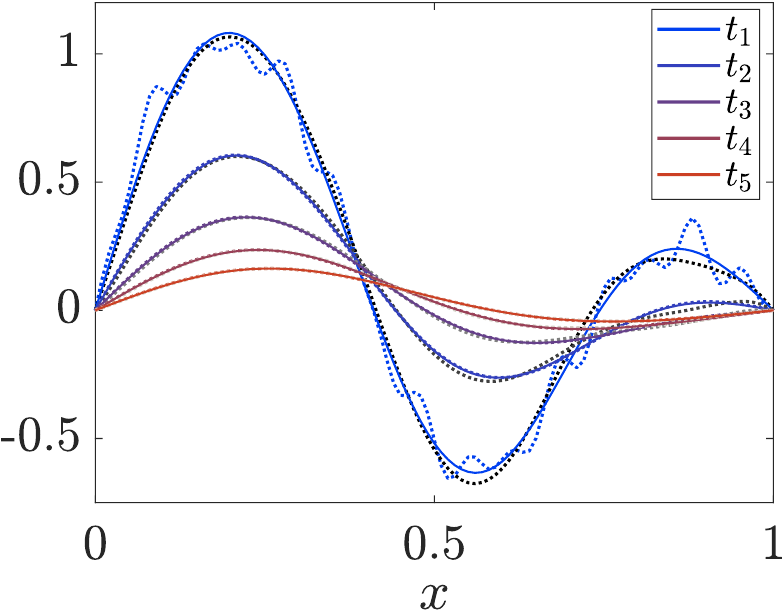}
    \end{minipage}
    \caption{(a)~Eigenfunctions $ \widetilde{\varphi}_\ell $ of the heat equation computed using exact functional DMD. The dotted lines represent the analytically computed eigenfunctions. (b)~Prediction of the dynamics using exact functional DMD, where $ t_i = (i-1) \tau $. The dotted lines in the same color represent the true solution and the gray dotted lines the projected DMD prediction. (c)~Forecasting for an oscillatory temperature profile that cannot be faithfully represented by the functions $ u_i $ or $ v_i $. Although the initial approximation is inaccurate, high frequencies are smoothed out quickly and the prediction improves over time.}
    \label{fig:heat equation exact}
\end{figure}

Let us consider again the heat equation defined in Example~\ref{ex:heat equation} and apply Algorithm~\ref{alg:exact functional DMD}. In order to compare exact and projected functional DMD, we reuse the training data constructed in Example~\ref{ex:heat equation projected}. The eigenfunctions $ \widetilde{\varphi}_\ell $, shown in Figure~\ref{fig:heat equation exact}\ts(a), are slightly more accurate than the projected DMD eigenfunctions presented in Figure~\ref{fig:heat equation projected}\ts(c). Furthermore, we predict the evolution of the temperature using projected and exact functional DMD for a new initial condition (projected onto $ \mathbb{U} $ and $ \mathbb{V} $, respectively). Both methods produce faithful forecasts as shown in Figure~\ref{fig:heat equation exact}\ts(b), exact DMD performs just slightly better. Note that the difference between the predicted and true dynamics decreases in time since higher-frequency terms are damped out. This can also be seen in Figure~\ref{fig:heat equation exact}\ts(c), where we choose a noisy initial condition that cannot be approximated well by the training data. Exact functional DMD is again a bit more accurate. \exampleSymbol
\end{example}

\subsubsection{Comparison with projected functional DMD}

One of the main differences between projected functional DMD and exact functional DMD is that the eigenfunctions $ \varphi_\ell $ are written in terms of the functions $ u_i $, whereas the eigenfunctions $ \widetilde{\varphi}_\ell $ are defined in terms of the functions $ v_i $. The eigenvalues, on the other hand, are identical since
\begin{equation*}
    \big(\Theta \ts \Sigma^{-1}\big)^{-1} A \ts \big(\Theta \ts \Sigma^{-1}\big)
        = \Sigma \ts \Theta^\top C_{uu}^{-1} \ts C_{uv} \ts \Theta \ts \Sigma^{-1}
        = \Sigma^{-1} \ts \Theta^\top C_{uv} \ts \Theta \ts \Sigma^{-1} = \widetilde{A},
\end{equation*}
i.e., the matrices $ A $ and $ \widetilde{A} $ are similar and $ \lambda_\ell = \widetilde{\lambda}_\ell $. If we restrict $ \widetilde{\mathcal{A}} $ to $ \mathbb{V} $, then, given a function $ g = V \beta $, the operator $ \widetilde{\mathcal{A}}\big|_\mathbb{V} \colon \mathbb{V} \to \mathbb{V} $ is defined by $ \widetilde{\mathcal{A}}\big|_\mathbb{V} \ts g = V (C_{uu}^{-1} \ts C_{uv} \ts \beta) $. In order to highlight the similarities between projected and exact functional DMD, we reformulate Algorithm~\ref{alg:exact functional DMD}, omitting the whitening transformation.

\begin{talgorithm}[Reformulated exact functional DMD] \label{alg:reformulated exact functional DMD}
The exact functional DMD eigenfunctions $ \widetilde{\varphi}_\ell $ can be computed as follows:
\begin{enumerate}[leftmargin=4ex, itemsep=0ex, topsep=0.5ex]
\item Define $ A = C_{uu}^{-1} \ts C_{uv} $.
\item Determine the eigenvalues $ \lambda_\ell $ and eigenvectors $ \xi^{(\ell)} $ of $ A $.
\item Compute $ \widetilde{\varphi}_\ell = \frac{1}{\widetilde{\lambda}_\ell \rule{0pt}{1.9ex}} V \ts \xi^{(\ell)} $.
\end{enumerate}
\end{talgorithm}

We have $ A \ts \xi^{(\ell)} = \Theta \ts \Sigma^{-2} \ts \Theta^\top C_{uv} \ts \xi^{(\ell)} = \lambda_\ell \ts \xi^{(\ell)} $ so that
\begin{equation*}
    \Sigma^{-1} \ts \Theta^\top C_{uv} \ts \xi^{(\ell)}
        = \Sigma \ts \Theta^\top \lambda_\ell \ts \xi^{(\ell)}
    \implies
 	\Sigma^{-1} \ts \Theta^\top C_{uv} \ts \Theta \ts \Sigma^{-1} \widetilde{\xi}^{(\ell)}
        = \lambda_\ell \ts \widetilde{\xi}^{(\ell)},
\end{equation*}
where $ \widetilde{\xi}^{(\ell)} = \Sigma \ts \Theta^\top \ts \xi^{(\ell)} $. Additionally, it holds that
\begin{equation*}
    \widetilde{\varphi}_\ell = \frac{1}{\widetilde{\lambda}_\ell \rule{0pt}{1.9ex}} V \ts \xi^{(\ell)} = \frac{1}{\widetilde{\lambda}_\ell \rule{0pt}{1.9ex}} V \ts \Theta \ts \Sigma^{-1} \widetilde{\xi}^{(\ell)}.
\end{equation*}
This shows that Algorithm~\ref{alg:exact functional DMD} and Algorithm~\ref{alg:reformulated exact functional DMD} are equivalent.

\begin{lemma}
Let $ \mathcal{R} $ denote the projection onto the space spanned by the left singular functions $ r_\ell $ of $ \mathcal{U} $, then $ \mathcal{R} \widetilde{\varphi}_\ell = \varphi_\ell $.
\end{lemma}

\begin{proof}
For a function $ g = V \beta $, the projection $ \mathcal{R} \colon \mathbb{H} \to \mathbb{H} $ is defined by $ \mathcal{R} g = U(C_{uu}^{-1} \ts C_{uv} \ts \beta) $. Given now an eigenfunction $ \widetilde{\varphi}_\ell = \frac{1}{\widetilde{\lambda}_\ell \rule{0pt}{1.9ex}} V \ts \xi^{(\ell)} $, this implies
\begin{equation*}
    \mathcal{R} \widetilde{\varphi}_\ell = \frac{1}{\widetilde{\lambda}_\ell \rule{0pt}{2.5ex}} U(C_{uu}^{-1} \ts C_{uv} \ts \xi^{(\ell)}) = U \xi^{(\ell)} = \varphi_\ell. \qedhere
\end{equation*}
\end{proof}

For conventional DMD, this was proven in \cite{TRLBK14}. We have therefore shown that the derived projected and exact functional DMD algorithms are in fact infinite-dimensional versions of the classical DMD counterparts, which can be obtained as special cases by choosing $ \mathbb{H} = \R^n $ and representing the functions $ u_i $ and $ v_i $ by vectors $ \mathbf{u}_i $ and $ \mathbf{v}_i $.

\subsection{Relationships between functional DMD and generalized EDMD}

Our functional DMD framework is also related to the generalized EDMD method proposed in \cite{Mauroy21}, which can be viewed as an extension of EDMD \cite{WKR15, KKS16} to nonlinear infinite-dimensional systems. Let $ \mathbb{F} = \{ \zeta \colon \mathbb{H} \to \R \} $ be the space of real-valued functionals, then the Koopman operator $ \mathcal{K}^\tau $ with lag time $ \tau $ associated with \eqref{eq:dynamical system} is defined by
\begin{equation*}
    \mathcal{K}^\tau \zeta(u) = \zeta(\phi^\tau(u)).
\end{equation*}
The Koopman operator for infinite-dimensional systems inherits some of the properties of the Koopman operator for ordinary differential equations. In particular, products of eigenfunctionals are again eigenfunctionals. Let $ \chi_{\ell_1} $ and $ \chi_{\ell_2} $ be eigenfunctionals corresponding to the eigenvalues $ \lambda_{\ell_1} $ and $ \lambda_{\ell_2} $, then
\begin{equation*}
    \mathcal{K}^\tau (\chi_{\ell_1} \chi_{\ell_2})(u) = \mathcal{K}^\tau \chi_{\ell_1}(u) \ts \mathcal{K}^\tau \chi_{\ell_2}(u) = \lambda_{\ell_1} \ts \chi_{\ell_1}(u) \ts \lambda_{\ell_2} \ts \chi_{\ell_2}(u) = \lambda_{\ell_1} \lambda_{\ell_2}(\chi_{\ell_1} \chi_{\ell_2})(u).
\end{equation*}
We specifically focus on linear operators. Assume that $ \big(e^{\tau \ts \mathcal{W}}\big)^* \widehat{\varphi}_\ell = \overline{\lambda}_\ell \ts \widehat{\varphi}_\ell $, i.e., $ \widehat{\varphi}_\ell $ is an eigenfunction of the adjoint of $ e^{\tau \ts \mathcal{W}} $, then $ \chi_\ell = \innerprod{\ts\cdot\ts}{\widehat{\varphi}_\ell} $ is an eigenfunctional of the Koopman operator since
\begin{equation*}
    \mathcal{K}^\tau \chi_\ell(u)
        = \innerprod{e^{\tau \ts \mathcal{W}} u}{\widehat{\varphi}_\ell}
        = \innerprod{u}{\big(e^{\tau \ts \mathcal{W}}\big)^* \widehat{\varphi}_\ell}
        = \innerprod{u}{\overline{\lambda}_\ell \ts \widehat{\varphi}_\ell}
        = \lambda_\ell \innerprod{u}{\widehat{\varphi}_\ell}
        = \lambda_\ell \ts \chi_\ell(u)
\end{equation*}
as also shown in \cite{Mauroy21}. We can thus construct infinitely many additional eigenfunctionals by computing products and powers of these principal eigenfunctionals.

Generalized EDMD projects the Koopman operator for infinite-dimensional systems onto a set of preselected basis functionals $ \{ \zeta_i \}_{i=1}^n $ by first computing the matrices $ Z_1, Z_2 \in \R^{n \times m} $, defined by
\begin{equation*}
    Z_1 =
    \begin{bmatrix}
        \zeta_1(u_1) & \dots & \zeta_1(u_m) \\
        \vdots & \ddots & \vdots \\
        \zeta_n(u_1) & \dots & \zeta_n(u_m)
    \end{bmatrix}
    \quad \text{and} \quad
    Z_2 =
    \begin{bmatrix}
        \zeta_1(v_1) & \dots & \zeta_1(v_m) \\
        \vdots & \ddots & \vdots \\
        \zeta_n(v_1) & \dots & \zeta_n(v_m)
    \end{bmatrix},
\end{equation*}
and then defining $ K^\top = Z_2 \ts Z_1^+ $. The eigenfunctionals of the projected operator are determined by the right eigenvectors of the matrix $ K $.

\begin{example}
We choose two different types of functionals:
\begin{enumerate}[leftmargin=4ex, itemsep=0ex, topsep=0.5ex, label=\roman*)]
\item Defining $ \zeta_i(u) = u(x_i) $ for the spatially discretized domain with grid points $ x_1, \dots, x_n $, it follows that $ Z_1 = U $ and $ Z_2 = V $ are the data matrices defined in Example~\ref{ex:KDE and grid discretization} so that $ K^\top = V \ts U^+ = B $, see also \cite{Mauroy21}. However, DMD computes the right eigenvectors of $ B $, which yields Koopman modes rather than Koopman eigenfunctions, cf.~\cite{KNKWKSN18}.
\item We now define $ n = m $ and $ \zeta_i(u) = \innerprod{u_i}{u} $ so that $ Z_1 = C_{uu} $ and $ Z_2 = C_{uv} $. It follows that $ K^\top = C_{uv} \ts C_{uu}^+ $, i.e., $ K = C_{uu}^+ \ts C_{vu} $, where $ C_{vu} = C_{uv}^\top $. The matrix $ K $ can be regarded as a Galerkin approximation of the operator $ \big(e^{\tau \ts \mathcal{W}}\big)^* $ since
\begin{equation*}
    \big[C_{vu}\big]_{ij} = \innerprod{v_i}{u_j} = \innerprod{e^{\tau \ts \mathcal{W}} u_i}{u_j} = \innerprod{u_i}{\big(e^{\tau \ts \mathcal{W}}\big)^* u_j}.
\end{equation*}
Assume that $ K \widehat{\xi}_\ell = \overline{\lambda}_\ell \ts \widehat{\xi}_\ell $, then $ \widehat{\varphi}_\ell = U \widehat{\xi}_\ell $ is an approximation of an eigenfunction of $ \big(e^{\tau \ts \mathcal{W}}\big)^* $ and $ \chi_\ell = \innerprod{\ts\cdot\ts}{\widehat{\varphi}_\ell} $ is an approximation of an eigenfunctional of the Koopman operator $ \mathcal{K}^\tau $. \exampleSymbol
\end{enumerate}
\end{example}

The functions $ u_i $ will in general not necessarily be a good basis for approximating eigenfunctions of the adjoint operator. Nevertheless, the comparison shows that functional DMD and generalized EDMD are closely related if we restrict ourselves to linear operators. Extensions of functional DMD to nonlinear operators will be considered in future work.

\section{Applications}
\label{sec:Applications}

In this section, we will highlight potential applications of the proposed functional DMD framework and present numerical results.

\subsection{Graphons}

A \emph{graphon} is a Lebesgue-measurable function $ w \colon [0, 1] \times [0, 1] \to [0, 1] $, where $ [0, 1] $ represents a continuum of vertices \cite{ LS06, Janson13, PLC21}. Two vertices $ x, y \in [0, 1] $ are connected by an edge with weight or probability $ w(x, y) $ if $ w(x, y) > 0 $ or unconnected if $ w(x, y) = 0 $. That is, $ w $ can be viewed as a generalization of a weighted adjacency matrix. A graphon is called \emph{symmetric} or \emph{undirected} if $ w(x, y) = w(y, x) $ for all $ x, y \in [0, 1] $. In what follows, we will only consider connected symmetric graphons.\!\footnote{Connectedness implies that any set $ A $ and its complement $ A^c $ are linked by an edge, see, e.g., \cite{PLC21, BPS22}.} Random walk processes are in this case reversible and associated transfer operators are self-adjoint w.r.t.\ suitably reweighted inner products. For a more detailed introduction, we refer to~\cite{KB26}.

\subsubsection{Transfer operators for discrete-time random walks}

The \emph{degree function} $ d \colon [0, 1] \to [0, 1] $ and the \emph{transition density function} $ p \colon [0, 1] \times [0, 1] \to [0, \infty) $ are defined by
\begin{equation*}
    d(x) = \int_0^1 w(x, y) \ts \mathrm{d}y
    \quad \text{and} \quad
    p(x, y) = \frac{w(x, y)}{d(x)}.
\end{equation*}
The unique \emph{invariant density} is then given by
\begin{equation*}
    \pi(x) = \frac{1}{Z} \ts d(x), \quad \text{with } Z = \int_0^1 d(x) \ts \mathrm{d}x.
\end{equation*}
We first consider transfer operators that describe the evolution of random walkers in discrete time.

\begin{definition}[Perron--Frobenius and Koopman operators]
Let $ w $ be a connected symmetric graphon with transition density function~$ p $.
\begin{enumerate}[leftmargin=4ex, itemsep=0ex, topsep=0.5ex, label=\roman*)]
\item We define the Koopman operator $ \mathcal{K} \colon L_{\pi}^2 \to L_{\pi}^2 $ by
\begin{equation*}
    \mathcal{K} f(x) = \int_0^1 p(x, y) \ts f(y) \ts \mathrm{d}y.
\end{equation*}
\item Analogously, we define the \emph{Perron--Frobenius operator} $ \mathcal{P} \colon L_{\nicefrac{1}{\pi}}^2 \to L_{\nicefrac{1}{\pi}}^2 $ by
\begin{equation*}
    \mathcal{P} \rho(x) = \int_0^1 p(y, x) \ts \rho(y) \ts \mathrm{d}y.
\end{equation*}
\end{enumerate}
\end{definition}

Given an eigenfunction $ \varphi_\ell $ of the Koopman operator, we can construct an eigenfunction of the Perron--Frobenius operator by defining $ \widehat{\varphi}_\ell = \pi \ts \varphi_\ell $, i.e., $ \mathcal{K} \varphi_\ell = \mu_\ell \ts \varphi_\ell \implies \mathcal{P} \widehat{\varphi}_\ell = \mu_\ell \ts \widehat{\varphi}_\ell $. It then holds that
\begin{align*}
    \mathcal{K} = \sum_\ell \mu_\ell \ts (\varphi_\ell \otimes \widehat{\varphi}_\ell)
    \quad \text{and} \quad
    \mathcal{P} = \sum_\ell \mu_\ell \ts (\widehat{\varphi}_\ell \otimes \varphi_\ell),
\end{align*}
which allows us to express the transition density function as well as the graphon itself in terms of the eigenfunctions, i.e.,
\begin{equation*}
    p(x, y) = \sum_\ell \mu_\ell \ts \varphi_\ell(x) \ts \widehat{\varphi}_\ell(y)
    \quad \text{and} \quad
    w(x, y) = Z \sum_\ell \mu_\ell \ts \widehat{\varphi}_\ell(x) \ts \widehat{\varphi}_\ell(y).
\end{equation*}
The normalization constant $ Z $ is in general unknown, it is only possible to reconstruct the graphon up to a multiplicative factor. This is due to the fact that multiplying a graphon by a constant does not affect the transition probabilities. Detailed derivations and proofs can be found in \cite{KB26}.

\subsubsection{Transfer operators for continuous-time random walks}

Instead of assuming that all the random walkers jump to another vertex at the same time, we now consider continuous-time random walks, where the waiting times are sampled from an exponential distribution. The associated continuous-time dynamics, which are closely related to the discrete-time counterparts, have been derived in \cite{PLC21}.

\begin{definition}[Rate operator]
The \emph{rate operator} $ \mathcal{Q} = \mathcal{K} - \mathcal{I} $ and its adjoint $ Q^* = \mathcal{P} - \mathcal{I} $ are defined by
\begin{equation*}
    \mathcal{Q} \ts f(x) = \int_0^1 p(x, y) \ts f(y) \ts \mathrm{d}y - f(x)
    \quad \text{and} \quad
    \mathcal{Q}^* \rho(x) = \int_0^1 p(y, x) \ts \rho(y) \ts \mathrm{d}y -\rho(x).
\end{equation*}
\end{definition}

The operator $ \mathcal{Q} $ can be regarded as a generalization of the rate matrix for a continuous-time Markov chain defined on a finite state space. Note that, given eigenfunctions $ \varphi_\ell $ of $ \mathcal{K} $ and $ \widehat{\varphi}_\ell $ of~$ \mathcal{P} $, it holds that $ \mathcal{Q} \varphi_\ell = (\mu_\ell - 1) \ts \varphi_\ell $ and $ \mathcal{Q}^* \widehat{\varphi}_\ell = (\mu_\ell - 1) \ts \widehat{\varphi}_\ell $. The evolution of observables and probability densities associated with the continuous-time random walk is then described by
\begin{equation} \label{eq:graphon dynamics}
    \frac{\raisebox{-2pt}{$\partial$}}{\partial t}f(x, t) = \mathcal{Q} \ts f(x, t)
    \quad \text{and} \quad
    \frac{\raisebox{-2pt}{$\partial$}}{\partial t} \ts \rho(x, t) = \mathcal{Q}^* \rho(x, t).
\end{equation}

\begin{example} \label{ex:triple-peak graphon}

\begin{figure}
    \centering
    \begin{minipage}[t]{0.32\linewidth}
        \centering
        \subfiguretitle{(a)}
        \includegraphics[height=3.9cm]{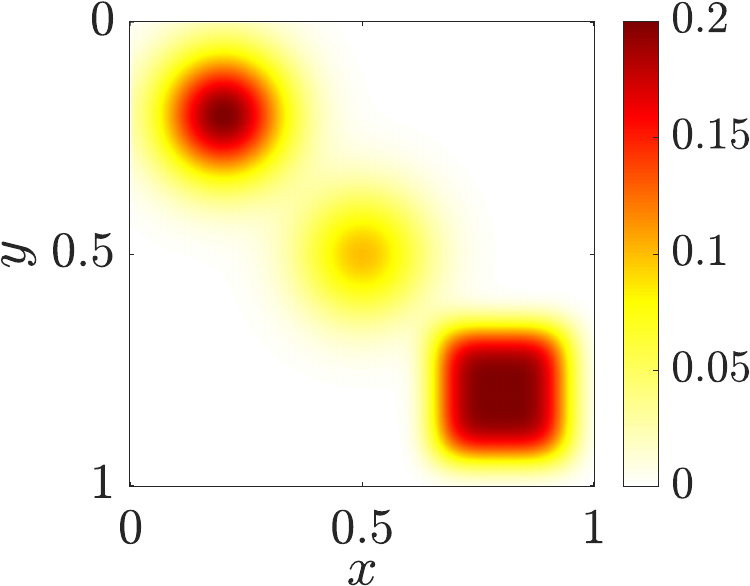}
    \end{minipage}
    \begin{minipage}[t]{0.32\linewidth}
        \centering
        \subfiguretitle{(b)}
        \includegraphics[height=3.9cm]{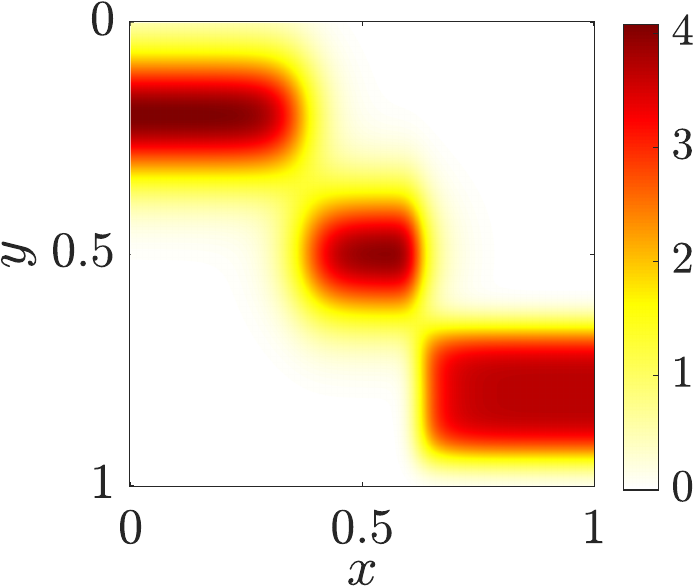}
    \end{minipage}
    \begin{minipage}[t]{0.32\linewidth}
        \centering
        \subfiguretitle{(c)}
        \includegraphics[height=3.95cm]{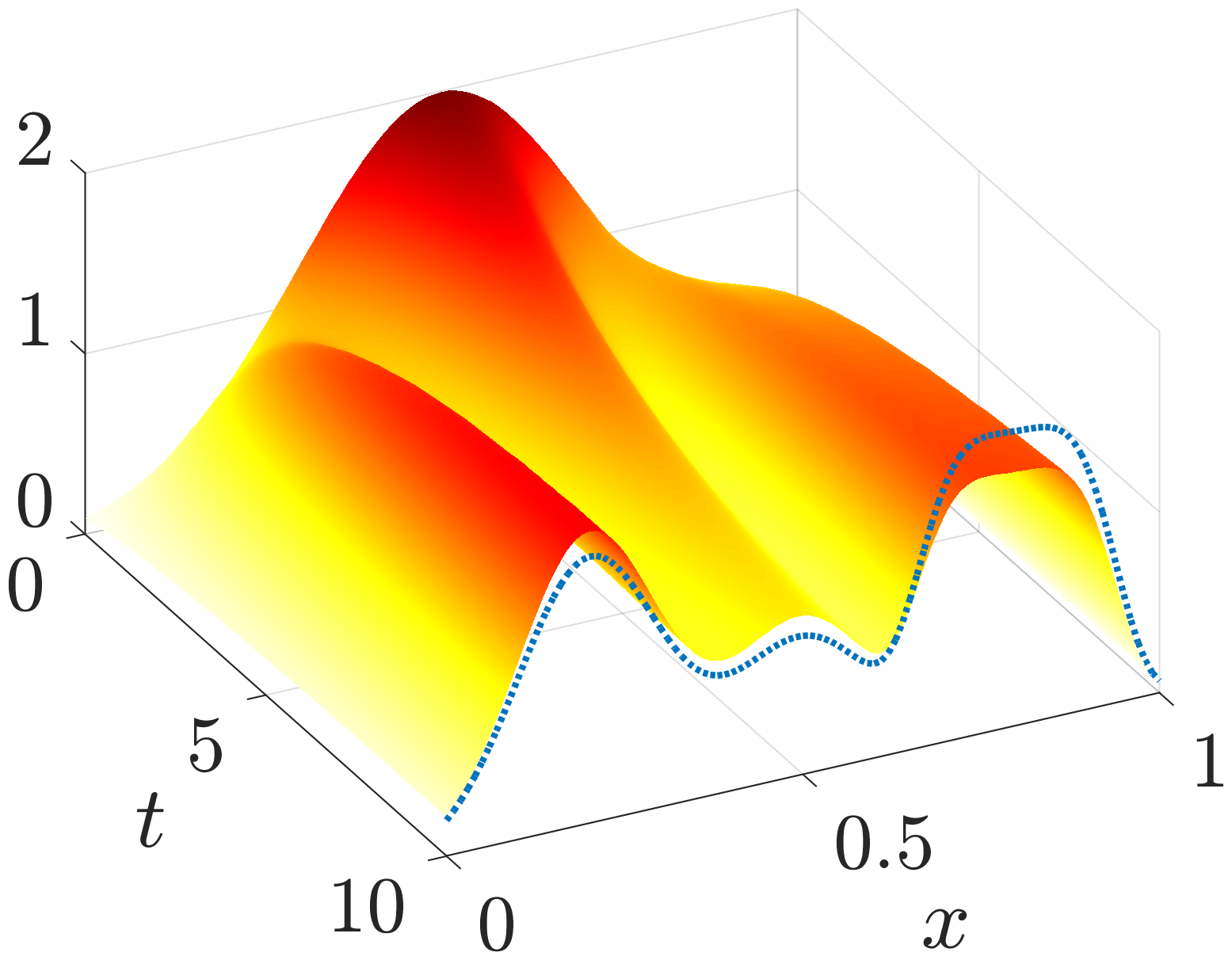}
    \end{minipage}
    \caption{(a)~Symmetric graphon $ w $ with three peaks at $ 0.2 $, $ 0.5 $, and $ 0.8 $. The peak in the middle is less metastable than the other two. (b)~Corresponding transition density function $ p $. (c)~Evolution of a probability density $ \rho $ in time. The initial density, a Gaussian with bandwidth $ \sigma = 0.2 $ centered at $ x = \frac{1}{2} $, spreads to the other clusters and converges to the invariant density, represented by the dotted blue line.}
    \label{fig:triple-peak graphon}
\end{figure}

In order to illustrate the continuous-time dynamics, we consider the graphon
\begin{equation*}
    w(x, y) =  0.2 \ts e^{-\frac{(x - 0.2)^2 + (y - 0.2)^2}{0.02}}
             + 0.1 \ts e^{-\frac{(x - 0.5)^2 + (y - 0.5)^2}{0.02}}
             + 0.2 \ts e^{-\frac{(x - 0.8)^4 + (y - 0.8)^4}{0.0005}},
\end{equation*}
shown in Figure~\ref{fig:triple-peak graphon}\ts(a), comprising three clusters~\cite{KB26}. The corresponding transition density function is visualized in Figure~\ref{fig:triple-peak graphon}\ts(b). Random walkers will typically spend a long time in one of the clusters before transitioning to a neighboring cluster. That is, the clusters form so-called metastable sets. Metastability implies that the process will appear to be almost equilibrated before transitioning to another part of the state space.\!\footnote{For more rigorous definitions, including relationships with spectral properties of transfer operators, see \cite{Davies82a, HS06, Bovier16}.} Given any initial density $ \rho_0 $, it will converge to the invariant density $ \pi \sim d $. This is shown in Figure~\ref{fig:triple-peak graphon}\ts(c). Due to the metastability of the random walk process (which manifests itself in eigenvalues of $ \mathcal{K} $ and $ \mathcal{P} $ close to one), the convergence is slow. \exampleSymbol
\end{example}

\begin{definition}[Graphon Laplacian]
We define the random-walk normalized \emph{graphon Laplacian} by $ \mathcal{L} = -\mathcal{Q} = \mathcal{I} - \mathcal{K} $ so that its adjoint is $ \mathcal{L}^* = -\mathcal{Q}^* = \mathcal{I} - \mathcal{P} $.
\end{definition}

The eigenvalues $ \mu_\ell $ of $ \mathcal{K} $ and $ \mathcal{P} $ are contained in the closed unit disk and, since we assume the graphon to be symmetric, also real-valued. The eigenvalues of $ \mathcal{L} $ and $ \mathcal{L}^* $ are hence contained in the interval $ [0, 2] $. Graph or graphon Laplacians are often used for spectral clustering and studying consensus problems \cite{Luxburg07, PLC21, BPS22, KB26}.

\subsubsection{Learning graphons from functional data}

Assuming we have only access to a time-evolving density of random walkers, but not the positions of the random walkers themselves, we show that it is still possible to detect clusters and to identify the graphon.

\begin{example}
\begin{figure}
    \definecolor{matlab1}{RGB}{0, 114, 189}
    \definecolor{matlab2}{RGB}{217, 83, 25}
    \definecolor{matlab3}{RGB}{237, 177, 32}
    \newcommand{\cdash}[1]{\textcolor{#1}{\rule[0.5ex]{1em}{0.2ex}}}
    \centering
    \begin{minipage}[t]{0.32\linewidth}
        \centering
        \subfiguretitle{(a)}
        \vspace*{0.8ex}
        \includegraphics[height=3.8cm]{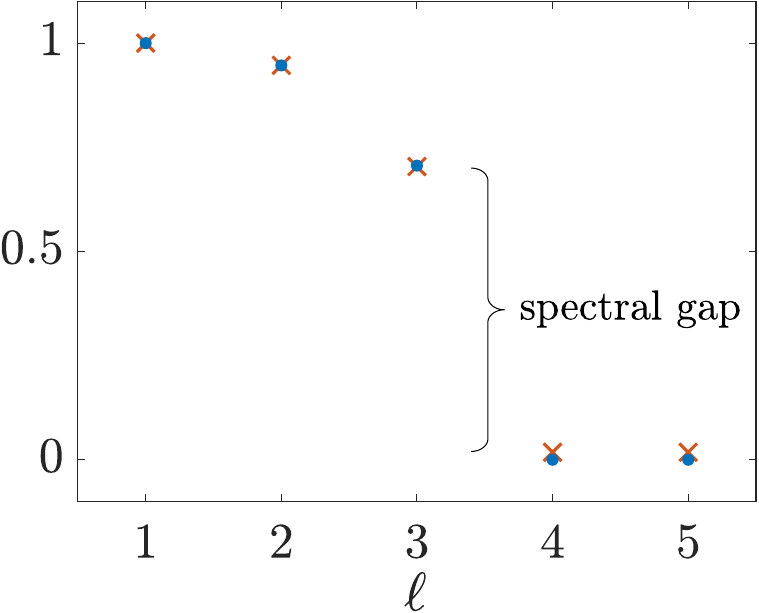}
    \end{minipage}
    \hspace{1ex}
    \begin{minipage}[t]{0.32\linewidth}
        \centering
        \subfiguretitle{(b)}
        \vspace*{0.8ex}
        \includegraphics[height=3.75cm]{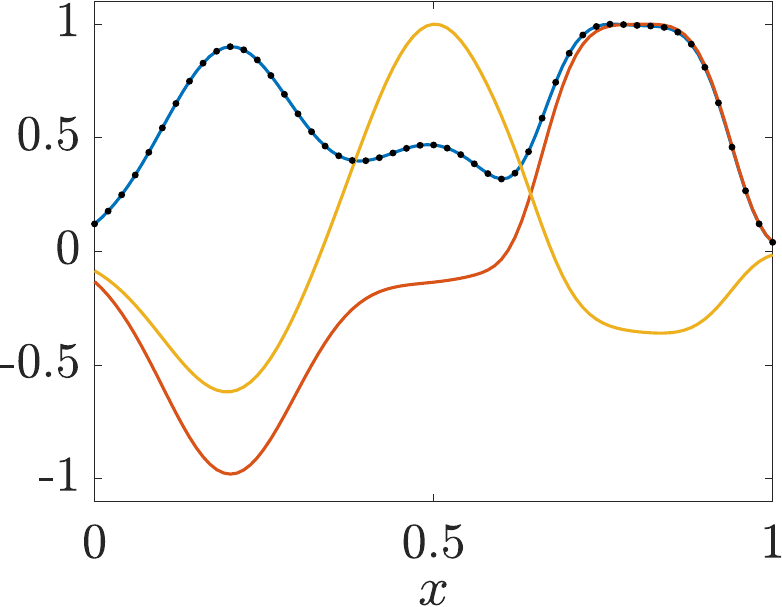}
    \end{minipage}
    \begin{minipage}[t]{0.32\linewidth}
        \centering
        \subfiguretitle{(c)}
        \includegraphics[height=3.9cm]{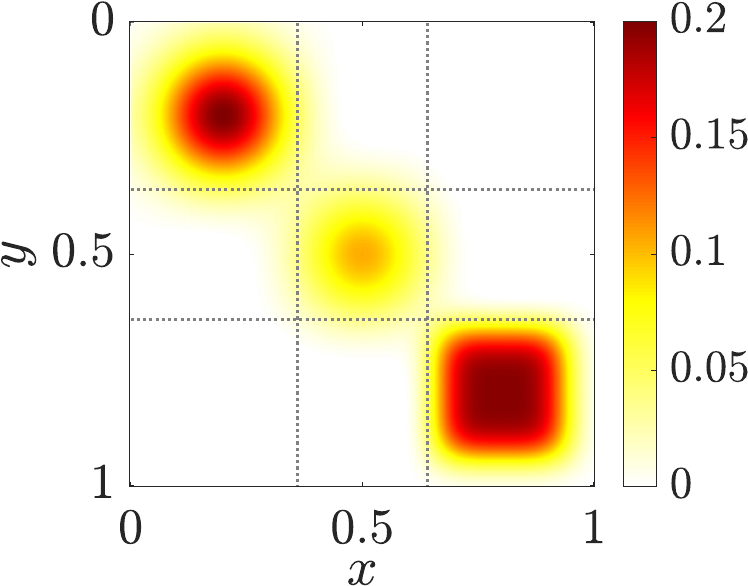}
    \end{minipage}
    \hspace{1ex}
    \caption{(a)~Dominant eigenvalues $ \mu_\ell $ of $ \mathcal{P} $. The red crosses, shown for comparison, are the eigenvalues estimated from one long discrete-time random walk. (b)~Eigenfunctions of $ \mathcal{P} $, where \cdash{matlab1} denotes the first, \cdash{matlab2} the second, and \cdash{matlab3} the third eigenfunction. The black dots represent the true invariant density $ \pi $. (c)~Rank-3 reconstruction of $ w $ using the estimated eigenfunctions. The resulting graphon is virtually indistinguishable from the true graphon shown in Figure~\ref{fig:triple-peak graphon}\ts(a). The dotted gray lines separate the identified clusters.}
    \label{fig:triple-peak results}
\end{figure}

Let us analyze the graphon introduced in Example~\ref{ex:triple-peak graphon}. We define the initial density $ \rho_0 $ to be a Gaussian with bandwidth $ \sigma = 0.2 $ centered at $ x = \frac{1}{2} $, see Figure~\ref{fig:triple-peak graphon}\ts(c), and simulate~\eqref{eq:graphon dynamics} from $ t = 0 $ to $ t = 5 $ using a lag time of $ \tau = 0.1 $ so that $ U $ and $ V $ contain 50 snapshots. We compute the matrices $ C_{uu} $ and $ C_{uv} $ and apply projected functional DMD. We then estimate the eigenvalues $ \mu_\ell $ of $ \mathcal{P} $, shown in Figure~\ref{fig:triple-peak results}\ts(a), from the approximated eigenvalues $ \lambda_\ell $ of $ e^{\tau \ts Q^*} $ using
\begin{equation*}
    \mu_\ell = \frac{\log(\lambda_\ell)}{\tau} + 1.
\end{equation*}
This allows us to compare the eigenvalues with the values obtained by considering discrete-time random walks, see \cite{KB26}. There are three dominant eigenvalues, followed by a spectral gap, implying the existence of three metastable sets. The corresponding eigenfunctions are shown in Figure~\ref{fig:triple-peak results}\ts(b). In order to detect clusters in the graphon, we apply $ k $-means with $ k = 3 $ to the dominant three eigenfunctions. The eigenfunctions can also be used to reconstruct the graphon itself, up to a multiplicative constant, as illustrated in Figure~\ref{fig:triple-peak results}\ts(c). Furthermore, we can now use the identified eigenfunctions to predict the evolution of the system. \exampleSymbol
\end{example}

The example demonstrates that we can extract the invariant density despite the fact that the simulation has not nearly reached it yet. Additionally, we can identify the graphon itself and forecast the dynamics using only functional data.

\subsection{Stochastic differential equations}

Although functional DMD can in the same way be applied to arbitrary autonomous stochastic differential equations, we will specifically consider Langevin dynamics. Let $ \mathbb{X} \subset \R^d $ be the state space. Given an energy potential $ W \colon \R^d \to \R $ and an inverse temperature $ \beta > 0 $, the overdamped Langevin equation is defined by
\begin{equation*}
    \mathrm{d}X_t = -\nabla W(X_t) \ts \mathrm{d}t + \sqrt{2 \ts \beta^{-1}} \ts \mathrm{d}B_t, \quad X_0 \sim \rho_0,
\end{equation*}
where $ B_t $ is a $ d $-dimensional Wiener process and $ \rho_0 $ is the initial density of $ X $. Depending on the potential and the inverse temperature, such systems often exhibit metastable behavior.

\subsubsection{Transfer operators for Langevin dynamics}

The evolution of observables $ f $ and probability densities $ \rho $ associated with the stochastic process is described by the \emph{Kolmogorov backward equation} and \emph{Fokker--Planck equation}, respectively, defined by
\begin{equation*}
    \frac{\raisebox{-2pt}{$\partial$}}{\partial t}f(x, t) = \mathcal{L} \ts f(x, t)
    \quad \text{and} \quad
    \frac{\raisebox{-2pt}{$\partial$}}{\partial t} \ts \rho(x, t) = \mathcal{L}^* \rho(x, t),
\end{equation*}
with
\begin{equation*}
    \mathcal{L} f = -\nabla W \vdot \nabla f + \beta^{-1} \Delta f
    \quad \text{and} \quad
    \mathcal{L}^* \rho = \Delta W \ts \rho + \nabla W \vdot \nabla \rho + \beta^{-1} \Delta \rho.
\end{equation*}
The corresponding propagators are the Koopman operator $ \mathcal{K}^\tau $ and Perron--Frobenius operator $ \mathcal{P}^\tau $, which are closely related to the same operators defined above for graphons. The invariant density (also called Gibbs or Boltzmann distribution) $ \pi \sim e^{-\beta W} $ of the overdamped Langevin equation satisfies $ \mathcal{L}^* \pi = 0 $ or, equivalently, $ \mathcal{P}^\tau \pi = \pi $. A detailed introduction to Langevin dynamics and transfer operators for stochastic differential equations can be found in \cite{LaMa94, SS13, Pav14, SKH23}.

\begin{example} \label{ex:Himmelblau}

\begin{figure}
    \centering
    \begin{minipage}[t]{0.32\linewidth}
        \centering
        \subfiguretitle{(a)}
        \includegraphics[height=4.1cm]{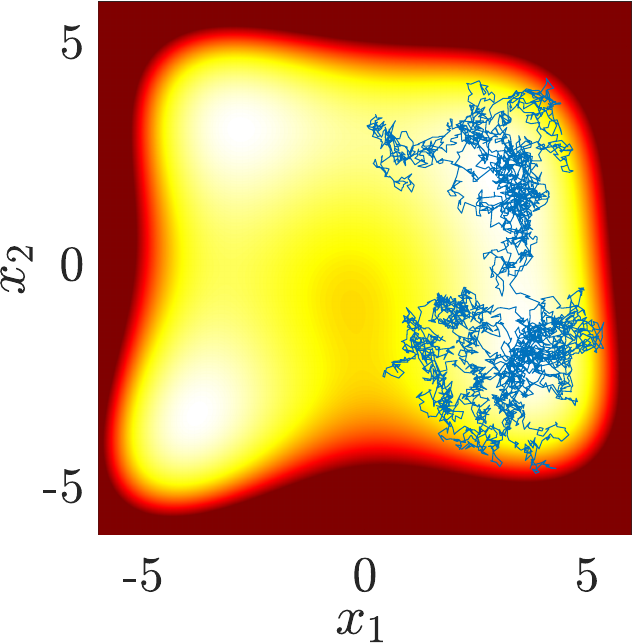}
    \end{minipage}
    \begin{minipage}[t]{0.32\linewidth}
        \centering
        \subfiguretitle{(b)}
        \includegraphics[height=4.1cm]{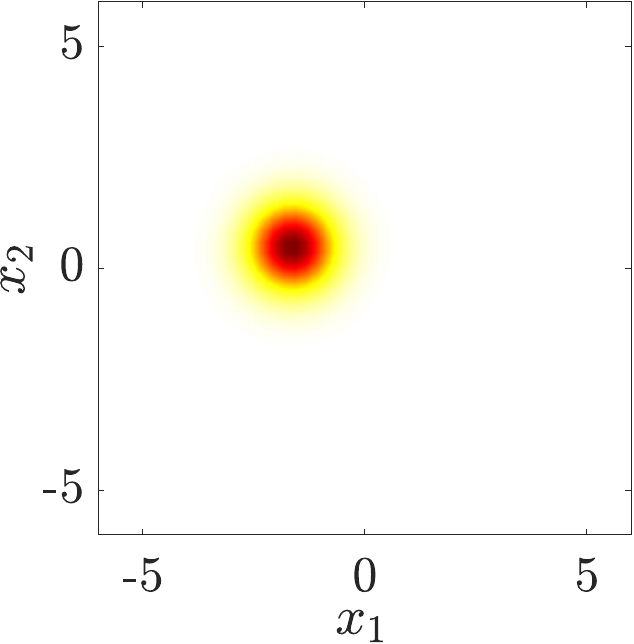}
    \end{minipage}
    \begin{minipage}[t]{0.32\linewidth}
        \centering
        \subfiguretitle{(c)}
        \includegraphics[height=4.1cm]{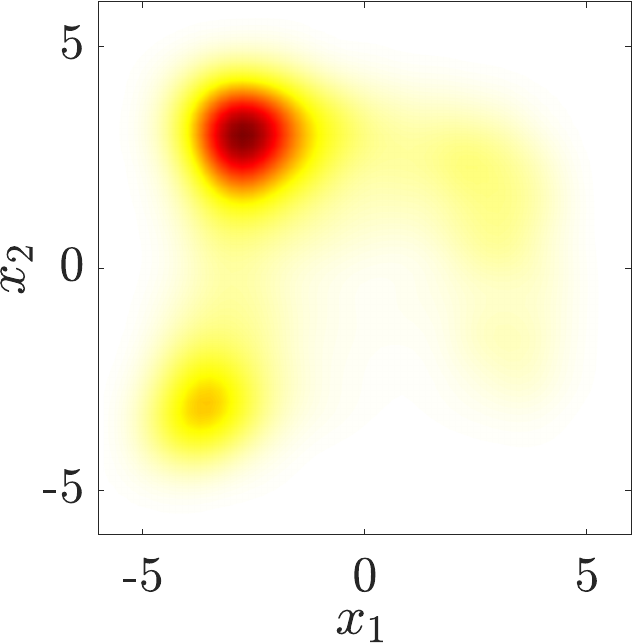}
    \end{minipage}
    \caption{(a)~Visualization of the Himmelblau potential comprising two separate wells in the left half plane and two partially merged wells in the right half plane. The blue line represents a single non-equilibrated trajectory. (b)~Initial density given by a kernel density estimate computed from 5000 initial conditions sampled from a Gaussian distribution with randomly generated center and bandwidth. (c)~Estimate of the density at time $ \tau $. The initial density spreads to the wells of the Himmelblau potential, but has clearly not reached the stationary distribution yet.}
    \label{fig:Himmelblau system}
\end{figure}

As a simple example, we consider the two-dimensional Himmelblau potential
\begin{equation*}
    W(x) = (x_1^2 + x_2 - 11)^2 + (x_1 + x_2^2 - 7)^2,
\end{equation*}
shown in Figure~\ref{fig:Himmelblau system}\ts(a), and choose the inverse temperature $ \beta = \frac{2}{100} $ and lag time $ \tau = \frac{1}{10} $, see also~\cite{SKH23}. Trajectories will typically spend a long time in one well before transitioning to one of the other wells. Since $ \beta $ is quite small and the two wells in the right half plane close to each other, they can be considered to form one large well. We would hence expect three dominant eigenvalues close to one, indicating the existence of three metastable sets, followed by a spectral gap. In order to generate training data for functional DMD, we sample $ 5000 $ points from a Gaussian distribution with randomly selected center and bandwidth and then apply kernel density estimation, described in Example~\ref{ex:KDE and grid discretization}, to construct the density at $ t = 0 $, see Figure~\ref{fig:Himmelblau system}\ts(b). The sampled points are mapped forward using the overdamped Langevin equation to obtain the time-lagged points, from which we estimate the density at $ t = \tau $, as illustrated in Figure~\ref{fig:Himmelblau system}\ts(c). \exampleSymbol
\end{example}

Alternatively, we could assume that we have access to densities at different time points, either obtained by applying a black-box PDE solver or by repeatedly measuring the densities of particles. The goal here is to illustrate the flexibility and versatility of functional DMD and in particular the kernel-based formulation. Gaussian mixture models might provide a more data-efficient alternative.

\subsubsection{Detecting invariant densities and metastable states}

Given only estimates of the densities, functional DMD allows us to approximate dominant eigenfunctions of the Perron--Frobenius operator, which in turn can be used to identify the invariant density, metastable sets, and also the potential itself since $ W \sim - \frac{1}{\beta} \log(\pi) $. Additionally, the eigenvalues contain information about the associated timescales and the number of metastable sets.

\begin{example}

\begin{figure}
    \centering
    \begin{minipage}[t]{0.32\linewidth}
        \centering
        \subfiguretitle{(a) $ \lambda_1 \approx 1 $}
        \includegraphics[height=4.1cm]{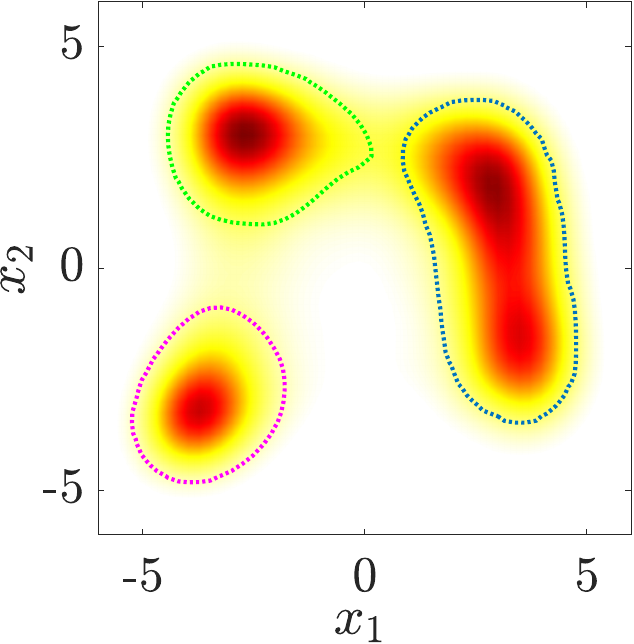}
    \end{minipage}
    \begin{minipage}[t]{0.32\linewidth}
        \centering
        \subfiguretitle{(b) $ \lambda_2 \approx 0.851 $}
        \includegraphics[height=4.1cm]{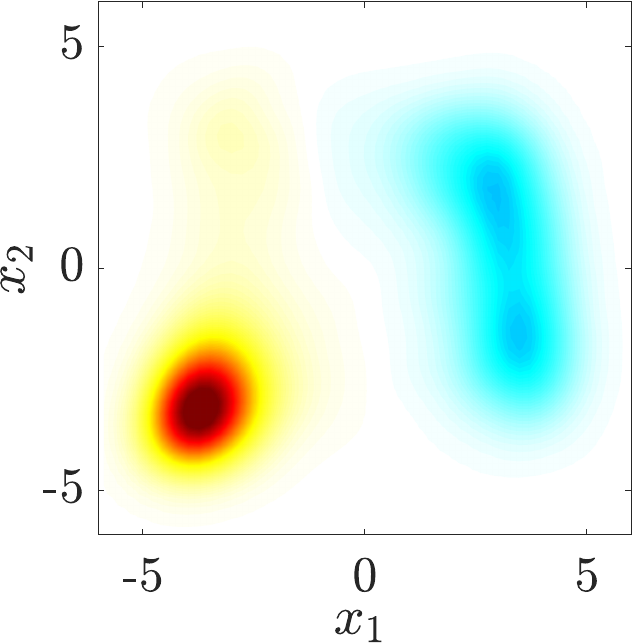}
    \end{minipage}
    \begin{minipage}[t]{0.32\linewidth}
        \centering
        \subfiguretitle{(c) $ \lambda_3 \approx 0.695 $}
        \includegraphics[height=4.1cm]{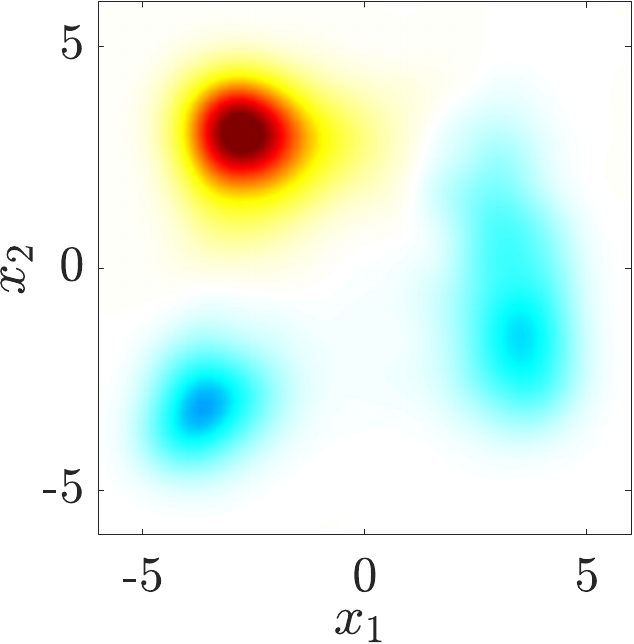}
    \end{minipage}
    \caption{(a) Estimated invariant density. The dotted lines mark the three identified metastable sets. (b)~The second eigenfunction separates the well in the lower left corner from the others. (c)~The third eigenfunction distinguishes between the well in the upper left corner and the other wells. Combining this information allows us to extract the three metastable sets shown in (a).}
    \label{fig:Himmelblau results}
\end{figure}

Let us consider again the Himmelblau system introduced in Example~\ref{ex:Himmelblau}. We collect training data by randomly generating five initial densities (Gaussians centered at uniformly sampled points in $ [-5, 5]^2 $ with bandwidth one), for which we compute the corresponding densities at times $ \tau $, $ 2 \ts \tau $, and $ 3 \ts \tau $, resulting in three time-lagged pairs, i.e., overall $ m = 15 $ functions $ u_i $ and~$ v_i $. We then compute the Gram matrices $ C_{uu}, C_{uv} \in \R^{15 \times 15} $ as described in Example~\ref{ex:KDE and grid discretization}, choosing a normalized Gaussian kernel
\begin{equation*}
    k(x, x') = \big(2 \ts \pi \ts \sigma^2\big)^{-\frac{d}{2}} \ts \exp\left(-\frac{\norm{x - x'}^2}{2 \ts \sigma^2}\right)
\end{equation*}
with bandwidth $ \sigma = \frac{1}{2} $ for the kernel density estimation, where $ d = 2 $ is the dimension of the system. Computing the eigenvalues of the matrix $ A $ reveals that there are indeed three metastable sets. The resulting projected functional DMD eigenfunctions are shown in Figure~\ref{fig:Himmelblau results}. We extract the metastable sets using the \emph{sparse eigenbasis approximation} (SEBA) algorithm~\cite{FRS19}. In order to analyze how accurate the computed eigenvalues are, we apply standard EDMD \cite{WKR15, KKS16} with a dictionary containing monomials of order up to eight directly to the SDE data.\!\footnote{A more suitable comparison would be to apply kernel EDMD \cite{WRK15, KSM19}. However, the data set contains $ 15 \times 5\ts000 $ points so that the resulting kernel matrices would be $ 75\ts000 $-dimensional.} We then obtain the eigenvalues $ \widehat{\lambda}_1 \approx 1 $, $ \widehat{\lambda}_2 \approx 0.854 $, and $ \widehat{\lambda}_3 \approx 0.699 $, which are close to the projected functional DMD estimates. \exampleSymbol
\end{example}

\begin{remark}
An extension of Ulam's method that approximates the transition kernel of the Perron--Frobenius operator with the aid of kernel density estimates instead of piecewise constant functions was proposed in \cite{SFB24}. We, on the other hand, represent the functional time-series data in terms of kernel density estimates and then approximate the operator itself using functional DMD.
\end{remark}

Functional DMD is conceptually different from conventional DMD-based methods in that it does not work with trajectory data generated by an ODE or SDE, but rather directly with functions whose evolution is described by a linear operator. In the example above, the densities are estimated from SDE data. One major difference though is that we do not need to be able to track individual trajectories but only aggregated properties of ensembles of particles such as their distributions.

\subsection{Koopman--von Neumann mechanics}

In addition to the classical transfer operators that propagate observables or probability densities, there exists a less well-known quantum physics-inspired formulation of classical mechanics, which describes the evolution of dynamical systems in terms of wavefunctions---the so-called Koopman--von Neumann equation \cite{Mauro02, Klein2018, Joseph20, KNG26}. We will now consider autonomous ordinary differential equations of the form $ \dot{x} = b(x) $, where $ b \colon \Omega \to \R^d $ and $ \Omega \subseteq \R^d $. If the domain is bounded, we will assume that $ \Omega $ is forward-invariant under the flow, which means that trajectories cannot leave the domain \cite{MauMez16}.

\subsubsection{The Koopman--von Neumann generator}

The evolution of observables $ f $, probability densities $ \rho $, and wavefunctions $ \psi $ can be described by the partial differential equations
\begin{equation*}
    \frac{\raisebox{-2pt}{$\partial$}}{\partial t}f(x, t) = \mathcal{L} f(x, t),  \qquad
    \frac{\raisebox{-2pt}{$\partial$}}{\partial t}\rho(x, t) = \mathcal{L}^* \rho(x, t), \qquad
    \frac{\raisebox{-2pt}{$\partial$}}{\partial t} \psi(x, t) = \mathcal{L}^\circ \psi(x, t),
\end{equation*}
where $ \mathcal{L} $ is the Koopman generator, $ \mathcal{L}^* $ the Perron--Frobenius generator, and $ \mathcal{L}^\circ $ the Koopman--von Neumann generator, defined by
\begin{equation*}
    \mathcal{L} f = b \vdot \nabla f,  \qquad
    \mathcal{L}^* \rho = -b \vdot \nabla \rho - \div(b) \ts \rho, \qquad
    \mathcal{L}^\circ \ts \psi = -b \vdot \nabla \psi - \tfrac{1}{2}\div(b) \ts \psi.
\end{equation*}
One main advantage of the Koopman--von Neumann generator is that it is skew-adjoint, which implies that the associated propagator for a fixed lag time $ \tau $ is unitary. Projecting this propagator onto a finite-dimensional state space, we obtain a unitary matrix, which can be represented by a quantum circuit. The Koopman--von Neumann framework can thus potentially be used to simulate classical dynamical systems on quantum computers.

\subsubsection{Linear systems and invariant subspaces}

It has been shown in \cite{KNG26} that for linear ordinary differential equations we can construct invariant subspaces, provided that a suitable conservation law can be found. Assume that $ \mathcal{L} \varphi_0 = 0 $ and $ \varphi_0 $ vanishes on $ \partial \Omega $, then the space
\begin{equation*}
    \mathbb{M} = \mspan\big\{ \varphi_0(x) \ts x^p : \abs{p} \le r \},
\end{equation*}
where $ p = (p_1, \dots, p_d) \in \mathbb{N}_0^d $ is a multi-index and $ \abs{p} = \sum_{i=1}^d p_i $, is invariant under the action of the three operators introduced above. In this case, it is possible to compute eigenvalues and eigenfunctions analytically. We will use such a system as a benchmark problem to assess the accuracy of functional DMD for unitary operators.

\begin{example} \label{ex:linear ODE}
Let us consider the system of linear ordinary differential equations $ \dot{x} = B \ts x $, with
\begin{equation*}
    B =
    \begin{bmatrix}
        0 & -1 & -1 \\
        1 &  0 & -1 \\
        1 &  1 &  0
    \end{bmatrix}.
\end{equation*}
We choose $ \varphi_0(x) = x_1^2 + x_2^2 + x_3^2 - 1 $ and define $ \Omega = \big\{ x_1^2 + x_2^2 + x_3^2 < 1 \big\} $ so that $ \mathcal{L} \varphi_0 = 0 $ and $ \varphi_0(x) = 0 $ on $ \partial \Omega $. The eigenvalues of the generator $ \mathcal{L} $ are determined by the eigenvalues of the matrix $ B $ and the eigenfunctions by the left eigenvectors. We obtain
\begin{alignat*}{2}
    \mu_1 &= \hspace{2em} 0, &\qquad  \varphi_1(x) &= x_1 - x_2 + x_3, \\
    \mu_2 &= -\mathrm{i} \ts \sqrt{3}, &\qquad \varphi_2(x) &= \big(\!-1 + \mathrm{i} \ts \sqrt{3}\big) x_1 + \big(1 + \mathrm{i} \ts \sqrt{3}\big) x_2 + 2 \ts x_3, \\
    \mu_3 &= \phantom{+} \mathrm{i} \ts \sqrt{3}, & \varphi_3(x) &= \big(\!-1 - \mathrm{i} \ts \sqrt{3}\big) x_1 + \big(1 - \mathrm{i} \ts \sqrt{3}\big) x_2 + 2 \ts x_3.
\end{alignat*}
Additional eigenfunctions can be constructed by computing products and powers of the principal eigenfunctions, i.e., $ \varphi_{(\ell_1, \ell_2, \ell_3)}(x) := \varphi_0(x) \ts \varphi_1(x)^{\ell_1} \varphi_2(x)^{\ell_2} \varphi_3(x)^{\ell_3} $ is an eigenfunction associated with the eigenvalue $ \mu_{(\ell_1, \ell_2, \ell_3)} = \ell_1 \ts \mu_1 + \ell_2 \ts \mu_2 + \ell_3 \ts \mu_3 = \mathrm{i} \ts \sqrt{3}(\ell_3 - \ell_2) $. Note in particular that different combinations of $ \ell_1 $, $ \ell_2 $, and $ \ell_3 $ correspond to the same eigenvalue. The conservation law $ \varphi_0 $ is used to enforce the Dirichlet boundary condition. The constructed functions, some of which are shown in Figure~\ref{fig:KvN example}\ts(a), are also eigenfunctions of the Perron--Frobenius and Koopman--von Neumann generator, corresponding to the eigenvalue $ -\mu_{(\ell_1, \ell_2, \ell_3)} $. \exampleSymbol
\end{example}

\subsubsection{Spectral decomposition and forecasting}

We are interested in estimating eigenvalues and eigenfunctions of the Koopman--von Neumann propagator from functional time-series data.

\begin{example}
\begin{figure}
    \centering
    \begin{minipage}[t]{0.45\linewidth}
        \centering
        \subfiguretitle{(a)}
        \vspace*{0.5ex}
        \includegraphics[height=6.0cm]{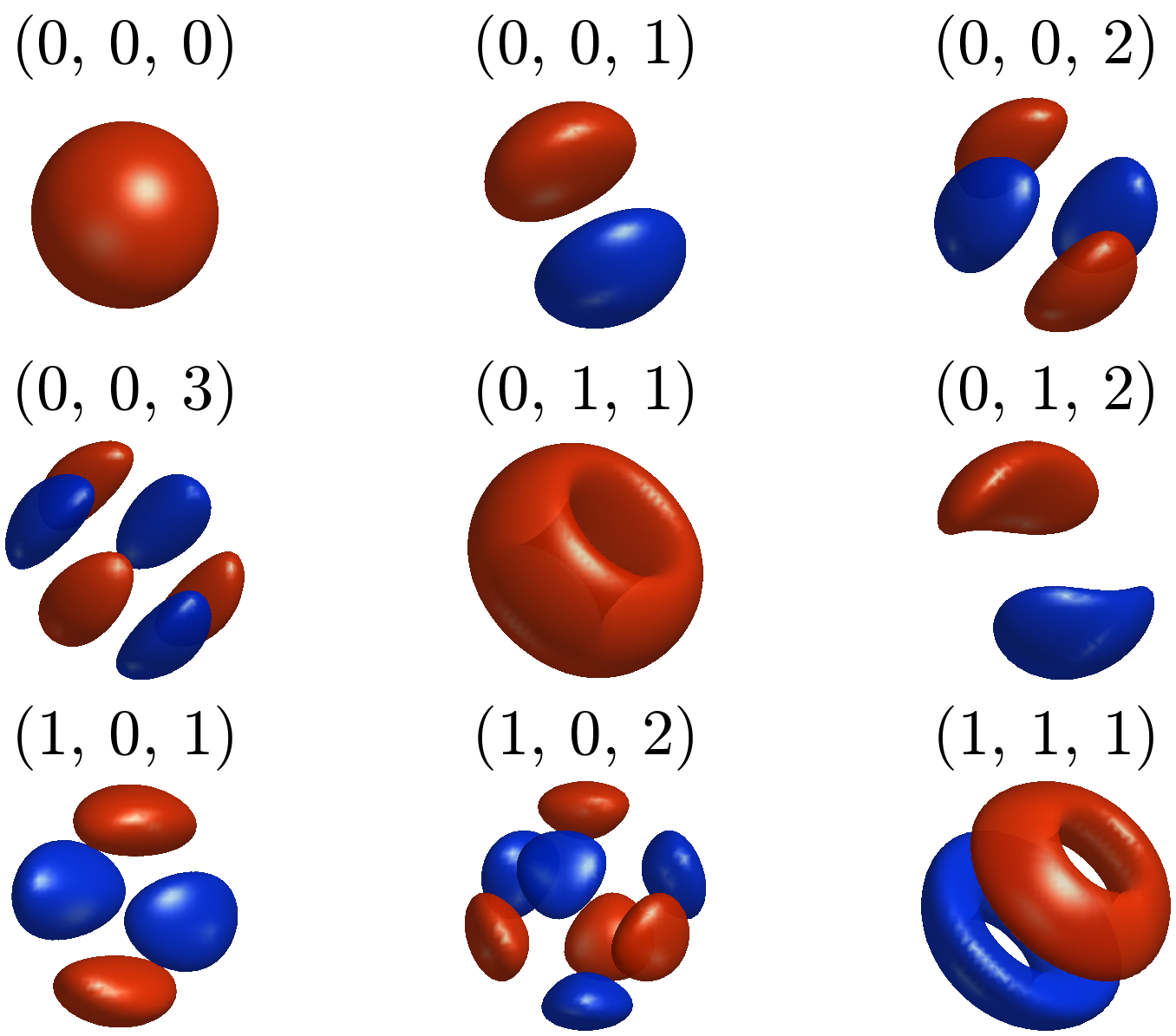}
    \end{minipage}
    \begin{minipage}[t]{0.45\linewidth}
        \centering
        \subfiguretitle{(b)}
        \vspace*{0.5ex}
        \includegraphics[height=6.2cm]{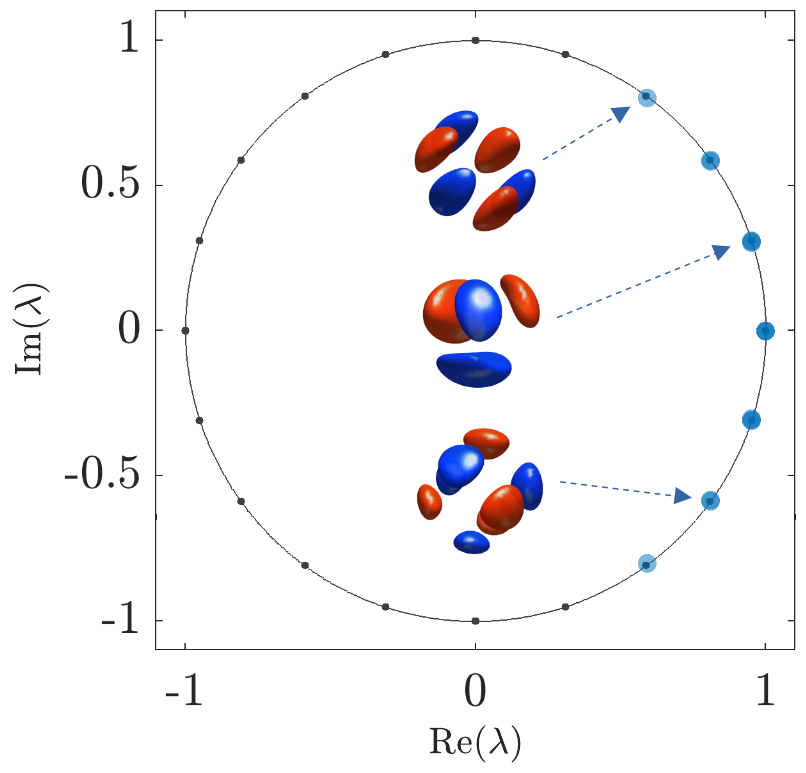}
    \end{minipage}
    \caption{(a) Analytically computed eigenfunctions of the Koopman--von Neumann generator associated with the linear system. The tuples $ (\ell_1, \ell_2, \ell_3) $ are the corresponding ``quantum numbers''. (b) Numerically (blue) and analytically (black) computed eigenvalues. A darker blue implies a higher multiplicity. Since the eigenspaces are not one-dimensional, eigenvectors and hence eigenfunctions are not uniquely determined. That is, the numerically computed eigenfunctions can be superpositions of the analytically computed eigenfunctions associated with the same eigenvalue.}
    \label{fig:KvN example}
\end{figure}

We apply exact functional DMD to the system introduced in Example~\ref{ex:linear ODE}. In order to generate training data, we simulate the Koopman--von Neumann equation (restricted to the $ 20 $-dimensional invariant subspace with $ r = 3 $) for five different randomly generated initial conditions and take three snapshot pairs with lag time $ \tau = \frac{2 \ts \pi}{20 \ts \sqrt{3}} $ from each simulation so that $ m = 15 $.  The computed eigenvalues, shown in Figure~\ref{fig:KvN example}\ts(b), lie on the unit circle. We can obtain estimates of the eigenvalues of the Koopman--von Neumann generator by computing
\begin{equation*}
    \widehat{\mu}_\ell = \frac{\log(\lambda_\ell)}{\tau}.
\end{equation*}
The estimated generator eigenvalues are approximately $ 0 $, $ \pm \mathrm{i} \ts \sqrt{3} $, $ \pm \mathrm{i} \ts 2 \sqrt{3} $, and $ \pm \mathrm{i} \ts 3 \sqrt{3} $, with multiplicities $ 3 $, $ 3 $, $ 2$, and $ 1 $. A few select eigenfunctions are also shown in Figure~\ref{fig:KvN example}\ts(b). Increasing the size of the dictionary would allow us to detect more of the eigenfunctions shown in Figure~\ref{fig:KvN example}\ts(a). \exampleSymbol
\end{example}

The propagator for the Koopman--von Neumann generator is unitary. Since all eigenvalues lie on the unit circle, the eigenfunctions represent non-decaying periodic patterns with different frequencies. The numerically computed spectral properties are good approximations of the analytically determined eigenvalues and eigenfunctions and can again be used for predicting the evolution of the system.

\subsection{Kuramoto--Sivashinsky equation}

As a last benchmark problem, we consider a nonlinear partial differential equation, namely the Kuramoto--Sivashinsky equation in two dimensions, which for spatially periodic domains can be defined by
\begin{equation*}
\frac{\raisebox{-2pt}{$\partial$}}{\partial t} u(x, t) = - \Delta u(x,t) - \Delta^2
u(x,t) - \tfrac{1}{2} \norm{\nabla u(x,t)}^2.
\end{equation*}
Although the interpretation of the eigenfunctions will be less clear since we approximate the nonlinear right-hand side by a linear operator, we can nevertheless apply functional DMD to the data, assuming that the estimated operator still contains relevant information about the global dynamics.

\begin{example}
\begin{figure}
    \centering
    \begin{minipage}[t]{0.64\linewidth}
        \centering
        \subfiguretitle{(a)}
        \includegraphics[height=5.8cm]{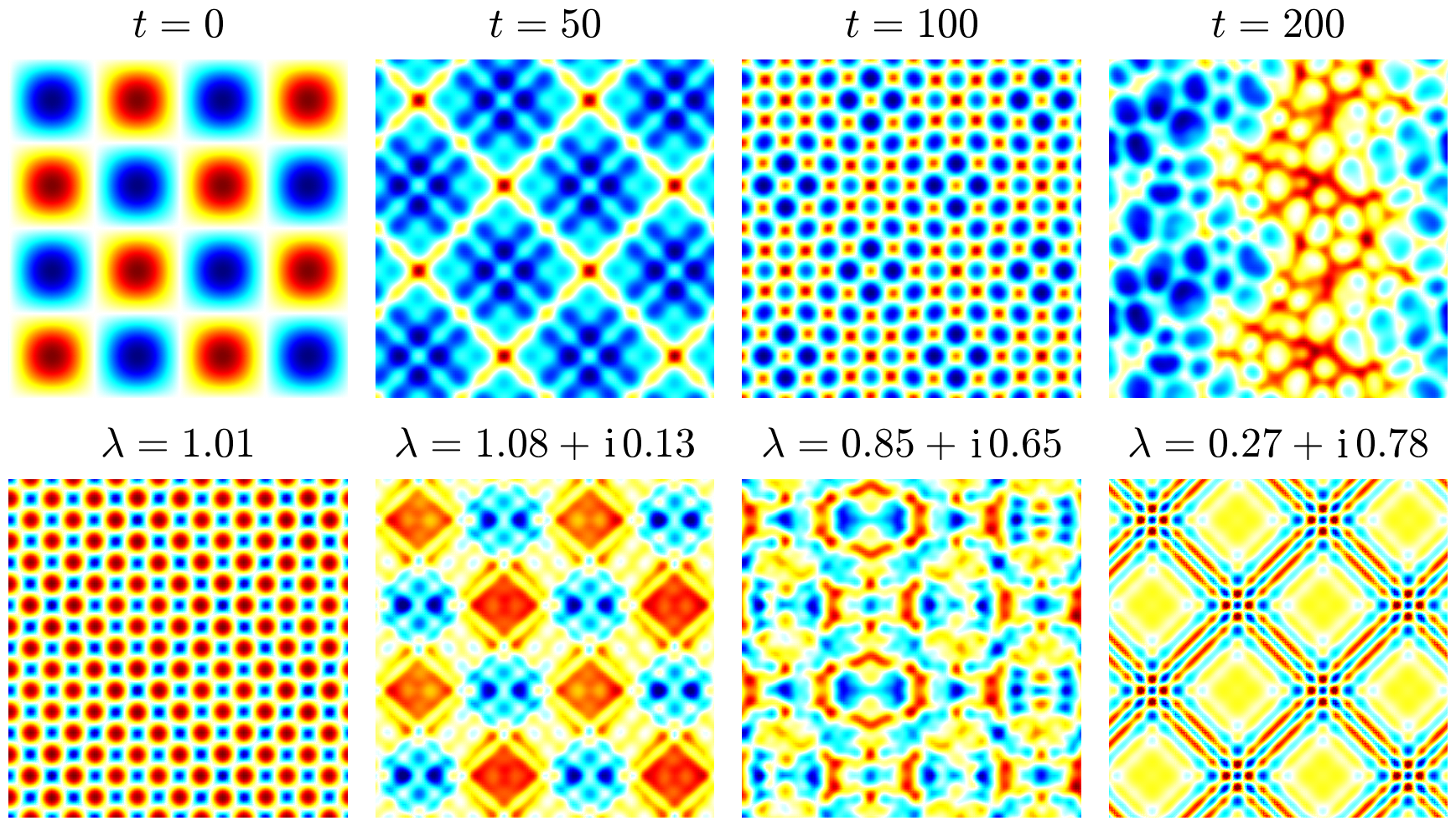}
    \end{minipage}
    \begin{minipage}[t]{0.35\linewidth}
        \centering
        \subfiguretitle{(b)}
        \vspace*{2.4ex}
        \includegraphics[height=5.5cm]{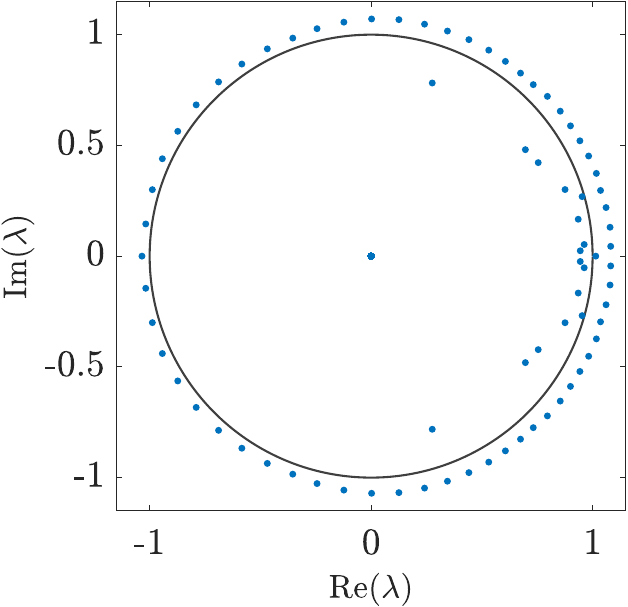}
    \end{minipage}
    \caption{(a) Top row: Solutions of the Kuramoto--Sivashinsky equation at different times $ t $. Bottom row: A few select numerically computed eigenfunctions of the estimated linear operator. (b) Spectrum of the operator. The eigenvalues do not lie on the unit circle in this case.}
    \label{fig:KS example}
\end{figure}

We choose the domain $ \Omega = (0, 30 \ts \pi) \times (0, 30 \ts \pi) $, periodic boundary conditions, a Fourier basis comprising $ 128 \times 128 $ functions, and the initial condition $ u_0(x) = \sin\big(\frac{2}{15} x_1\big) \sin\big(\frac{2}{15} x_2\big) $ and then simulate the Kuramoto--Sivashinsky equation from $ t = 0 $ to $ t = 200 $ using \emph{Shenfun} \cite{Shen94, Mortensen18}, a Python package containing spectral Galerkin methods for solving partial differential equations. We select $ \tau = 1 $ and $ t_i = (i - 1) \ts \tau $ so that we obtain $ 200 $ snapshots $ u_i $ and $ v_i $, a few of which are shown in the top row of Figure~\ref{fig:KS example}\ts(a), and then apply exact functional DMD. Four of the resulting eigenfunctions are displayed in the bottom row of Figure~\ref{fig:KS example}\ts(a). The DMD eigenvalues are shown in Figure~\ref{fig:KS example}\ts(b). \exampleSymbol
\end{example}

DMD or, equivalently, time-lagged independent component analysis \cite{MS94, PPGDN13} are often used as a preprocessing step in order to project high-dimensional data onto low-dimensional subspaces in such a way that the slowest timescales of the system are preserved.\!\footnote{Note that this is different from a PCA-based projection, which maximizes the variance of the projected data and does not explicitly take the temporal ordering of the snapshots into account.} A few modes or time-lagged independent components typically already capture the characteristic behavior---such as metastability---of complex multiscale systems. In the same way, we can now project infinite-dimensional data onto the slowly evolving dynamics using functional DMD. In the projected space, we can then, for instance, construct Markov state models or apply manifold learning techniques.

\section{Conclusion}
\label{sec:Conclusion}

We derived, analyzed, and compared two different DMD variants that can be used to learn infinite-dimensional dynamical systems and to identify dominant spatiotemporal patterns, namely \emph{projected functional DMD} and \emph{exact functional DMD}. Instead of estimating finite-dimensional matrices from vector-valued observations, we learn finite-rank operators from functional data. We have shown that by discretizing the domain and evaluating the training functions in select grid points, we obtain the well-known classical DMD algorithms as special cases. Although DMD has of course already been applied to infinite-dimensional problems, e.g., partial differential equations or integro-differential equations, the systems were typically first implicitly discretized and turned into finite-dimensional problems. We have in particular shown how functional DMD can be used to approximate infinite-dimensional transfer operators associated with ODEs, SDEs, and graphons. This is different from the typically considered particle-based point of view, where we assume that individual trajectories are given. By working directly with observables, densities, or wavefunctions, there is no need to track particles. This is, for instance, advantageous if we cannot distinguish between particles and have only density estimates.

Our DMD variants provide an abstract and flexible framework for learning operators from functional data, allowing us to work with arbitrary Hilbert spaces. Instead of relying on the standard Euclidean inner product, it is possible to leverage higher-order numerical integration techniques or to derive kernel-based methods. We have demonstrated that accurate estimates of dominant eigenfunctions can be obtained from just a few observations, using simple guiding examples such as the heat equation as well as random-walk processes on graphons, Langevin dynamics, Koopman--von Neumann mechanics, and the Kuramoto--Sivashinsky equation. Functional DMD could also shed light on the convergence of classical DMD algorithms if we consider the limit of infinitely many grid points. An open question is what happens in the infinite number of snapshots limit. Another issue might be the unavoidable curse of dimensionality: How can we efficiently represent or decompose functions if the state space is high-dimensional? One possibility would be to consider kernels defined on Hilbert spaces~\cite{WD22}, functional tensor trains \cite{GKM19}, or functional neural networks~\cite{RR23}. An interesting avenue for future research would also be to extend the proposed algorithms to nonlinear infinite-dimensional systems and to compare the resulting methods with generalized EDMD \cite{MauGon16}. Just like other DMD-type algorithms, functional DMD will in general produce spurious eigenvalues. A score that measures how trustworthy eigenvalues are was proposed in \cite{Colbrook23} and could be extended to the functional DMD setting. Incorporating domain knowledge into the learning process---e.g., conservation laws, symmetries, or the fact that the operator is self-adjoint or unitary---could also improve the accuracy and efficiency of functional DMD.

\section*{Data availability}

The code that supports the findings presented in this paper is available at \href{https://github.com/sklus/d3s/}{github.com/sklus/d3s/}.

\section*{No-AI disclaimer}

The authors did not use generative AI or AI-assisted technologies in their research or preparation of this manuscript.

\section*{Acknowledgments}

We thank Alex Mauroy and Stefanie Winkelmann for interesting discussions about graphons, interacting particle systems, and operator learning. S.K.\ was funded by a Leverhulme Trust Research Fellowship. E.I.\ was supported by the EPSRC Centre for Doctoral Training in Mathematical Modelling, Analysis and Computation (MAC-MIGS) funded by the UK Engineering and Physical Sciences Research Council (grant EP/S023291/1).

\bibliographystyle{unsrturl}
\bibliography{fDMD}

\begin{thebibliography}{10}

\bibitem{LNVR07}
D.~J. Levitin, R.~L. Nuzzo, B.~W. Vines, and J.~O. Ramsay.
\newblock Introduction to functional data analysis.
\newblock {\em Canadian Psychology}, 48(3):135--155, 2007.

\bibitem{Shang14}
H.~L. Shang.
\newblock A survey of functional principal component analysis.
\newblock {\em AStA Advances in Statistical Analysis}, 98(2):121--142, 2014.

\bibitem{EH15}
R.~Eubank and T.~Hsing.
\newblock {\em Theoretical Foundations of Functional Data Analysis with an
  Introduction to Linear Operators}.
\newblock Wiley, Chichester, 1st edition, 2015.

\bibitem{WCM16}
J.-L. Wang, J.-M. Chiou, and H.-G. Müller.
\newblock Functional data analysis.
\newblock {\em Annual Review of Statistics and Its Application}, 3:257--295,
  2016.
\newblock \href {https://doi.org/10.1146/annurev-statistics-041715-033624}
  {\path{doi:10.1146/annurev-statistics-041715-033624}}.

\bibitem{Schmid10}
P.~J. Schmid.
\newblock Dynamic mode decomposition of numerical and experimental data.
\newblock {\em Journal of Fluid Mechanics}, 656:5--28, 2010.
\newblock \href {https://doi.org/10.1017/S0022112010001217}
  {\path{doi:10.1017/S0022112010001217}}.

\bibitem{TRLBK14}
J.~H. Tu, C.~W. Rowley, D.~M. Luchtenburg, S.~L. Brunton, and J.~N. Kutz.
\newblock On dynamic mode decomposition: {T}heory and applications.
\newblock {\em Journal of Computational Dynamics}, 1(2), 2014.

\bibitem{WKR15}
M.~O. Williams, I.~G. Kevrekidis, and C.~W. Rowley.
\newblock A data-driven approximation of the {K}oopman operator: Extending
  dynamic mode decomposition.
\newblock {\em Journal of Nonlinear Science}, 25(6):1307--1346, 2015.
\newblock \href {https://doi.org/10.1007/s00332-015-9258-5}
  {\path{doi:10.1007/s00332-015-9258-5}}.

\bibitem{KNPNCS20}
S.~Klus, F.~N\"uske, S.~Peitz, J.-H. Niemann, C.~Clementi, and C.~Sch\"utte.
\newblock Data-driven approximation of the {K}oopman generator: {M}odel
  reduction, system identification, and control.
\newblock {\em Physica D: Nonlinear Phenomena}, 406:132416, 2020.
\newblock \href {https://doi.org/10.1016/j.physd.2020.132416}
  {\path{doi:10.1016/j.physd.2020.132416}}.

\bibitem{KNKWKSN18}
S.~Klus, F.~N\"uske, P.~Koltai, H.~Wu, I.~Kevrekidis, C.~Sch\"utte, and
  F.~No\'e.
\newblock Data-driven model reduction and transfer operator approximation.
\newblock {\em Journal of Nonlinear Science}, 28:985--1010, 2018.
\newblock \href {https://doi.org/10.1007/s00332-017-9437-7}
  {\path{doi:10.1007/s00332-017-9437-7}}.

\bibitem{Ko31}
B.~O. Koopman.
\newblock Hamiltonian systems and transformations in {H}ilbert space.
\newblock {\em Proceedings of the National Academy of Sciences}, 17(5):315,
  1931.
\newblock \href {https://doi.org/10.1073/pnas.17.5.315}
  {\path{doi:10.1073/pnas.17.5.315}}.

\bibitem{KvN32}
B.~O. Koopman and J.~von Neumann.
\newblock Dynamical systems of continuous spectra.
\newblock {\em Proceedings of the National Academy of Sciences of the United
  States of America}, 18(3):255--263, 1932.

\bibitem{LaMa94}
A.~Lasota and M.~C. Mackey.
\newblock {\em Chaos, fractals, and noise: Stochastic aspects of dynamics},
  volume~97 of {\em Applied Mathematical Sciences}.
\newblock Springer, New York, 2nd edition, 1994.

\bibitem{Mezic05}
I.~Mezi{\'{c}}.
\newblock Spectral properties of dynamical systems, model reduction and
  decompositions.
\newblock {\em Nonlinear Dynamics}, 41(1):309--325, 2005.
\newblock \href {https://doi.org/10.1007/s11071-005-2824-x}
  {\path{doi:10.1007/s11071-005-2824-x}}.

\bibitem{BMM12}
M.~Budi{\v s}i{\'c}, R.~Mohr, and I.~Mezi{\'c}.
\newblock Applied {K}oopmanism.
\newblock {\em Chaos: An Interdisciplinary Journal of Nonlinear Science},
  22(4), 2012.
\newblock \href {https://doi.org/10.1063/1.4772195}
  {\path{doi:10.1063/1.4772195}}.

\bibitem{SS13}
C.~Sch\"utte and M.~Sarich.
\newblock {\em Metastability and Markov State Models in Molecular Dynamics:
  Modeling, Analysis, Algorithmic Approaches}.
\newblock Number~24 in Courant Lecture Notes. American Mathematical Society,
  2013.

\bibitem{MauGon16}
A.~Mauroy and J.~Goncalves.
\newblock Linear identification of nonlinear systems: {A} lifting technique
  based on the {K}oopman operator.
\newblock In {\em 2016 IEEE 55th Conference on Decision and Control (CDC)},
  pages 6500--6505, 2016.
\newblock \href {https://doi.org/10.1109/CDC.2016.7799269}
  {\path{doi:10.1109/CDC.2016.7799269}}.

\bibitem{KKS16}
S.~Klus, P.~Koltai, and C.~Sch{\"u}tte.
\newblock On the numerical approximation of the {P}erron--{F}robenius and
  {K}oopman operator.
\newblock {\em Journal of Computational Dynamics}, 3(1):51--79, 2016.
\newblock \href {https://doi.org/10.3934/jcd.2016003}
  {\path{doi:10.3934/jcd.2016003}}.

\bibitem{MMS20}
A.~Mauroy, I.~Mezi{\'c}, and Y.~Susuki, editors.
\newblock {\em The Koopman Operator in Systems and Control: Concepts,
  Methodologies, and Applications}.
\newblock Lecture Notes in Control and Information Sciences. Springer
  International Publishing, 2020.
\newblock \href {https://doi.org/10.1007/978-3-030-35713-9}
  {\path{doi:10.1007/978-3-030-35713-9}}.

\bibitem{WuNo20}
H.~Wu and F.~No{\'e}.
\newblock Variational approach for learning {M}arkov processes from time series
  data.
\newblock {\em Journal of Nonlinear Science}, 30:33--66, 2020.
\newblock \href {https://doi.org/10.1007/s00332-019-09567-y}
  {\path{doi:10.1007/s00332-019-09567-y}}.

\bibitem{KNG26}
S.~Klus, F.~Nüske, and P.~Gelß.
\newblock Numerical approximation of the {K}oopman--von {N}eumann equation:
  Operator learning and quantum computing, 2026.
\newblock \href {http://arxiv.org/abs/2604.08414} {\path{arXiv:2604.08414}}.

\bibitem{KD24}
S.~Klus and N.~Djurdjevac Conrad.
\newblock Dynamical systems and complex networks: A {K}oopman operator
  perspective.
\newblock {\em Journal of Physics: Complexity}, 5(4):041001, 2024.
\newblock \href {https://doi.org/10.1088/2632-072X/ad9e60}
  {\path{doi:10.1088/2632-072X/ad9e60}}.

\bibitem{Colbrook24}
M.~J. Colbrook.
\newblock The multiverse of dynamic mode decomposition algorithms.
\newblock In Siddhartha Mishra and Alex Townsend, editors, {\em Numerical
  Analysis Meets Machine Learning}, volume~25 of {\em Handbook of Numerical
  Analysis}, pages 127--230. Elsevier, 2024.
\newblock \href {https://doi.org/https://doi.org/10.1016/bs.hna.2024.05.004}
  {\path{doi:https://doi.org/10.1016/bs.hna.2024.05.004}}.

\bibitem{Mauroy21}
A.~Mauroy.
\newblock Koopman operator framework for spectral analysis and identification
  of infinite-dimensional systems.
\newblock {\em Mathematics}, (19), 2021.
\newblock \href {https://doi.org/10.3390/math9192495}
  {\path{doi:10.3390/math9192495}}.

\bibitem{OTY25}
M.~Oprea, A.~Townsend, and Y.~Yang.
\newblock The distributional {K}oopman operator for random dynamical systems.
\newblock {\em Mathematics of Control, Signals, and Systems}, 37:769–798,
  2025.
\newblock \href {https://doi.org/10.1007/s00498-025-00423-x}
  {\path{doi:10.1007/s00498-025-00423-x}}.

\bibitem{KG20}
A.~Karimi and T.~T. Georgiou.
\newblock Data-driven approximation of the {P}erron--{F}robenius operator using
  the {W}asserstein metric.
\newblock {\em IFAC-PapersOnLine}, 55(30):341--346, 2022.
\newblock 25th International Symposium on Mathematical Theory of Networks and
  Systems MTNS 2022.
\newblock \href {https://doi.org/10.1016/j.ifacol.2022.11.076}
  {\path{doi:10.1016/j.ifacol.2022.11.076}}.

\bibitem{DHZ16}
M.~Dellnitz, M.~Hessel-Von Molo, and A.~Ziessler.
\newblock On the computation of attractors for delay differential equations.
\newblock {\em Journal of Computational Dynamics}, 3(1):93--112, 2016.
\newblock \href {https://doi.org/10.3934/jcd.2016005}
  {\path{doi:10.3934/jcd.2016005}}.

\bibitem{ZDG18}
A.~Ziessler, M.~Dellnitz, and R.~Gerlach.
\newblock The numerical computation of unstable manifolds for infinite
  dimensional dynamical systems by embedding techniques.
\newblock {\em SIAM Journal on Applied Dynamical Systems}, 18:1265--1292, 2018.
\newblock \href {https://doi.org/10.1137/18m1204395}
  {\path{doi:10.1137/18m1204395}}.

\bibitem{PHNPSW25}
S.~Peitz, H.~Harder, F.~N\"uske, F.~Philipp, M.~Schaller, and K.~Worthmann.
\newblock Equivariance and partial observations in {K}oopman operator theory
  for partial differential equations.
\newblock {\em Journal of Computational Dynamics}, 12(2):305--324, 2025.
\newblock \href {https://doi.org/10.3934/jcd.2024035}
  {\path{doi:10.3934/jcd.2024035}}.

\bibitem{RBPK17}
S.~H. Rudy, S.~L. Brunton, J.~L. Proctor, and J.~N. Kutz.
\newblock Data-driven discovery of partial differential equations.
\newblock {\em Science Advances}, 3(4):e1602614, 2017.
\newblock \href {https://doi.org/10.1126/sciadv.1602614}
  {\path{doi:10.1126/sciadv.1602614}}.

\bibitem{Pazy83}
A.~Pazy.
\newblock {\em Semigroups of linear operators and applications to partial
  differential equations}.
\newblock Springer, 1983.

\bibitem{EHN96}
H.~Engl, M.~Hanke, and A.~Neubauer.
\newblock {\em Regularization of Inverse Problems}.
\newblock Kluwer, Dordrecht, 1996.

\bibitem{MSKS20}
M.~Mollenhauer, I.~Schuster, S.~Klus, and C.~Sch\"utte.
\newblock Singular value decomposition of operators on reproducing kernel
  {H}ilbert spaces.
\newblock In {\em Advances in Dynamics, Optimization and Computation}, pages
  109--131, Cham, 2020. Springer.
\newblock \href {https://doi.org/10.1007/978-3-030-51264-4_5}
  {\path{doi:10.1007/978-3-030-51264-4_5}}.

\bibitem{LS06}
L.~Lovász and B.~Szegedy.
\newblock Limits of dense graph sequences.
\newblock {\em Journal of Combinatorial Theory, Series B}, 96(6):933--957,
  2006.
\newblock \href {https://doi.org/10.1016/j.jctb.2006.05.002}
  {\path{doi:10.1016/j.jctb.2006.05.002}}.

\bibitem{Janson13}
S.~Janson.
\newblock {\em Graphons, cut norm and distance, couplings and rearrangements},
  volume~4 of {\em New York Journal of Mathematics}.
\newblock State University of New York, University at Albany, Albany, NY, 2013.

\bibitem{PLC21}
J.~Petit, R.~Lambiotte, and T.~Carletti.
\newblock Random walks on dense graphs and graphons.
\newblock {\em SIAM Journal on Applied Mathematics}, 81(6):2323--2345, 2021.
\newblock \href {https://doi.org/10.1137/20M1339246}
  {\path{doi:10.1137/20M1339246}}.

\bibitem{BPS22}
B.~Bonnet, N.~{Pouradier Duteil}, and M.~Sigalotti.
\newblock Consensus formation in first-order graphon models with time-varying
  topologies.
\newblock {\em Mathematical Models and Methods in Applied Sciences},
  32(11):2121--2188, 2022.
\newblock \href {https://doi.org/10.1142/S0218202522500518}
  {\path{doi:10.1142/S0218202522500518}}.

\bibitem{KB26}
S.~Klus and J.~J. Bramburger.
\newblock Learning graphons from data: Random walks, transfer operators, and
  spectral clustering.
\newblock {\em IEEE Transactions on Signal Processing}, 74:1477--1490, 2026.
\newblock \href {https://doi.org/10.1109/TSP.2026.3682885}
  {\path{doi:10.1109/TSP.2026.3682885}}.

\bibitem{Davies82a}
E.~B. Davies.
\newblock Metastable states of symmetric {M}arkov semigroups {I}.
\newblock {\em Proceedings of the London Mathematical Society},
  s3-45(1):133--150, 1982.
\newblock \href {https://doi.org/10.1112/plms/s3-45.1.133}
  {\path{doi:10.1112/plms/s3-45.1.133}}.

\bibitem{Davies82b}
E.~B. Davies.
\newblock Metastable states of symmetric {M}arkov semigroups {II}.
\newblock {\em Journal of the London Mathematical Society}, s2-26(3):541--556,
  1982.
\newblock \href {https://doi.org/10.1112/jlms/s2-26.3.541}
  {\path{doi:10.1112/jlms/s2-26.3.541}}.

\bibitem{HS06}
W.~Huisinga and B.~Schmidt.
\newblock Metastability and dominant eigenvalues of transfer operators.
\newblock In {\em New Algorithms for Macromolecular Simulation}, volume~49 of
  {\em Lecture Notes in Computational Science and Engineering}, chapter~11,
  pages 167--182. Springer-Verlag, 2006.

\bibitem{Bovier16}
A.~Bovier and F.~{den Hollander}.
\newblock {\em Metastability: A Potential-Theoretic Approach}.
\newblock Grundlehren der mathematischen Wissenschaften. Springer International
  Publishing, 2016.

\bibitem{Luxburg07}
U.~von Luxburg.
\newblock A tutorial on spectral clustering.
\newblock {\em Statistics and Computing}, 17(4):395--416, 2007.
\newblock \href {https://doi.org/10.1007/s11222-007-9033-z}
  {\path{doi:10.1007/s11222-007-9033-z}}.

\bibitem{Pav14}
G.~A. Pavliotis.
\newblock {\em Stochastic Processes and Applications: Diffusion Processes, the
  Fokker--Planck and Langevin Equations}, volume~60 of {\em Texts in Applied
  Mathematics}.
\newblock Springer, 2014.

\bibitem{SKH23}
C.~Schütte, S.~Klus, and C.~Hartmann.
\newblock Overcoming the timescale barrier in molecular dynamics: {T}ransfer
  operators, variational principles and machine learning.
\newblock {\em Acta Numerica}, 32:517–673, 2023.
\newblock \href {https://doi.org/10.1017/S0962492923000016}
  {\path{doi:10.1017/S0962492923000016}}.

\bibitem{FRS19}
G.~Froyland, C.~P. Rock, and K.~Sakellariou.
\newblock Sparse eigenbasis approximation: Multiple feature extraction across
  spatiotemporal scales with application to coherent set identification.
\newblock {\em Communications in Nonlinear Science and Numerical Simulation},
  77:81--107, 2019.
\newblock \href {https://doi.org/10.1016/j.cnsns.2019.04.012}
  {\path{doi:10.1016/j.cnsns.2019.04.012}}.

\bibitem{WRK15}
M.~O. Williams, C.~W. Rowley, and I.~G. Kevrekidis.
\newblock A kernel-based method for data-driven {K}oopman spectral analysis.
\newblock {\em Journal of Computational Dynamics}, 2(2):247--265, 2015.
\newblock \href {https://doi.org/10.3934/jcd.2015005}
  {\path{doi:10.3934/jcd.2015005}}.

\bibitem{KSM19}
S.~Klus, I.~Schuster, and K.~Muandet.
\newblock Eigendecompositions of transfer operators in reproducing kernel
  {H}ilbert spaces.
\newblock {\em Journal of Nonlinear Science}, 2019.
\newblock \href {https://doi.org/10.1007/s00332-019-09574-z}
  {\path{doi:10.1007/s00332-019-09574-z}}.

\bibitem{SFB24}
S.~Surasinghe, J.~Fish, and E.~M. Bollt.
\newblock Learning transfer operators by kernel density estimation.
\newblock {\em Chaos: An Interdisciplinary Journal of Nonlinear Science},
  34(2):023126, 2024.
\newblock \href {https://doi.org/10.1063/5.0179937}
  {\path{doi:10.1063/5.0179937}}.

\bibitem{Mauro02}
D.~Mauro.
\newblock On {K}oopman--von {N}eumann waves.
\newblock {\em International Journal of Modern Physics A}, 17(09):1301--1325,
  2002.
\newblock \href {https://doi.org/10.1142/S0217751X02009680}
  {\path{doi:10.1142/S0217751X02009680}}.

\bibitem{Klein2018}
U.~Klein.
\newblock From {K}oopman--von {N}eumann theory to quantum theory.
\newblock {\em Quantum Studies: Mathematics and Foundations}, 5(2):219--227,
  2018.
\newblock \href {https://doi.org/10.1007/s40509-017-0113-2}
  {\path{doi:10.1007/s40509-017-0113-2}}.

\bibitem{Joseph20}
I.~Joseph.
\newblock Koopman--von {N}eumann approach to quantum simulation of nonlinear
  classical dynamics.
\newblock {\em Physical Review Research}, 2:043102, 2020.
\newblock \href {https://doi.org/10.1103/PhysRevResearch.2.043102}
  {\path{doi:10.1103/PhysRevResearch.2.043102}}.

\bibitem{MauMez16}
A.~Mauroy and I.~Mezi{\'{c}}.
\newblock Global stability analysis using the eigenfunctions of the {K}oopman
  operator.
\newblock {\em IEEE Transactions on Automatic Control}, 61(11):3356--3369,
  2016.
\newblock \href {https://doi.org/10.1109/TAC.2016.2518918}
  {\path{doi:10.1109/TAC.2016.2518918}}.

\bibitem{Shen94}
J.~Shen.
\newblock {Efficient Spectral-Galerkin Method I. Direct Solvers of Second- and
  Fourth-Order Equations Using Legendre Polynomials}.
\newblock {\em SIAM Journal on Scientific Computing}, 15(6):1489--1505, 1994.
\newblock \href {https://doi.org/10.1137/0915089} {\path{doi:10.1137/0915089}}.

\bibitem{Mortensen18}
M.~Mortensen.
\newblock Shenfun: High performance spectral {G}alerkin computing platform.
\newblock {\em Journal of Open Source Software}, 3(31):1071, 2018.
\newblock \href {https://doi.org/10.21105/joss.01071}
  {\path{doi:10.21105/joss.01071}}.

\bibitem{MS94}
L.~Molgedey and H.~G. Schuster.
\newblock Separation of a mixture of independent signals using time delayed
  correlations.
\newblock {\em Physical Review Letters}, 72:3634--3637, 1994.

\bibitem{PPGDN13}
G.~P{\'e}rez-Hern{\'a}ndez, F.~Paul, T.~Giorgino, G.~{De Fabritiis}, and
  F.~No{\'e}.
\newblock Identification of slow molecular order parameters for {M}arkov model
  construction.
\newblock {\em The Journal of Chemical Physics}, 139(1), 2013.

\bibitem{WD22}
G.~Wynne and A.~B. Duncan.
\newblock A kernel two-sample test for functional data.
\newblock {\em Journal of Machine Learning Research}, 23(73):1--51, 2022.
\newblock URL: \url{http://jmlr.org/papers/v23/20-1180.html}.

\bibitem{GKM19}
A.~Gorodetsky, S.~Karaman, and Y.~Marzouk.
\newblock A continuous analogue of the tensor-train decomposition.
\newblock {\em Computer Methods in Applied Mechanics and Engineering},
  347:59--84, 2019.
\newblock \href {https://doi.org/10.1016/j.cma.2018.12.015}
  {\path{doi:10.1016/j.cma.2018.12.015}}.

\bibitem{RR23}
A.~R. Rao and M.~Reimherr.
\newblock Nonlinear functional modeling using neural networks.
\newblock {\em Journal of Computational and Graphical Statistics},
  32(4):1248--1257, 2023.
\newblock \href {https://doi.org/10.1080/10618600.2023.2165498}
  {\path{doi:10.1080/10618600.2023.2165498}}.

\bibitem{Colbrook23}
M.~J. Colbrook.
\newblock The {mpEDMD} algorithm for data-driven computations of
  measure-preserving dynamical systems.
\newblock {\em SIAM Journal on Numerical Analysis}, 61(3):1585--1608, 2023.
\newblock \href {https://doi.org/10.1137/22M1521407}
  {\path{doi:10.1137/22M1521407}}.

\end{thebibliography}

\end{document}